\documentclass{amsart}
\usepackage{mathrsfs}
\usepackage{stmaryrd,mathtools}
\usepackage{enumerate}
\usepackage{tikz-cd}
\usepackage[all]{xy}
\usepackage{aliascnt}
\usepackage{bm}

\usepackage{shuffle}
\usepackage{ytableau}

\usepackage{enumitem}

\usepackage[colorlinks=true, linkcolor=blue, citecolor=blue, urlcolor=blue]{hyperref}
\usepackage{fullpage}
\usepackage{verbatim}
\usepackage{amssymb}

\newtheorem{theorem}{Theorem}[section]

\newaliascnt{lemma}{theorem}
\newtheorem{lemma}[lemma]{Lemma}
\aliascntresetthe{lemma}

\newaliascnt{corollary}{theorem}
\newtheorem{corollary}[corollary]{Corollary}
\aliascntresetthe{corollary}

\newaliascnt{proposition}{theorem}
\newtheorem{proposition}[proposition]{Proposition}
\aliascntresetthe{proposition}

\newaliascnt{potato}{theorem}

\aliascntresetthe{potato}

\newaliascnt{definitionlemma}{theorem}

\aliascntresetthe{definitionlemma}

\newaliascnt{conjecture}{theorem}

\aliascntresetthe{conjecture}

\newaliascnt{question}{theorem}

\aliascntresetthe{question}

\theoremstyle{definition}

\newaliascnt{definition}{theorem}

\aliascntresetthe{definition}

\newaliascnt{remark}{theorem}
\newtheorem{remark}[remark]{Remark}
\aliascntresetthe{remark}

\newaliascnt{example}{theorem}
\newtheorem{examplex}[example]{Example}
\newenvironment{example}
{\pushQED{\qed}\examplex}
{\popQED\endexamplex}
\aliascntresetthe{example}

\newaliascnt{notation}{theorem}

\aliascntresetthe{notation}

\newcommand{\spref}[1]{\href{https://stacks.math.columbia.edu/tag/#1}{#1}}

\usepackage{tikz}

\newcommand{\cB}{\mathcal B}
\newcommand{\cC}{\mathcal C}
\newcommand{\cD}{\mathcal D}

\newcommand{\cG}{\mathcal G}

\newcommand{\cI}{\mathcal I}

\newcommand{\cL}{\mathcal L}
\newcommand{\cM}{\mathcal M}
\newcommand{\cN}{\mathcal N}
\newcommand{\cO}{\mathcal O}
\newcommand{\cP}{\mathcal P}

\newcommand{\cU}{\mathcal U}
\newcommand{\cV}{\mathcal V}

\newcommand{\cX}{\mathcal X}
\newcommand{\cY}{\mathcal Y}
\newcommand{\cZ}{\mathcal Z}

\newcommand{\fkm}{\mathfrak m}

\newcommand{\fkp}{\mathfrak p}

\newcommand{\bA}{\mathbb A}

\newcommand{\bG}{\mathbb G}

\newcommand{\bP}{\mathbb P}

\newcommand{\bZ}{\mathbb Z}

\newcommand{\td}{\widetilde}
\newcommand{\ov}{\overline}

\DeclareMathOperator{\charac}{char}
\DeclareMathOperator{\codim}{codim}
\DeclareMathOperator{\colim}{colim}

\DeclareMathOperator{\sh}{sh}

\DeclareMathOperator{\Frac}{Frac}

\DeclareMathOperator{\Mor}{Mor}

\DeclareMathOperator{\Pic}{Pic}

\DeclareMathOperator{\Spec}{Spec}

\makeatletter
\@namedef{subjclassname@2020}{%
	\textup{2020} Mathematics Subject Classification}
\makeatother

\begin{document}
	
	\title{Resolution of indeterminacy of rational maps to proper tame stacks}
	
    \author{Myeong Jae Jeon}
	\thanks{}
	\address[MJJ]{Department of Mathematics, University of Maryland, College Park MD 20742, USA}
	\email{mjjeon@umd.edu}
	
	\date{\today}
	\keywords{}
	\subjclass[2025]{}

	\begin{abstract} 
        We show the resolution of indeterminacy of rational maps from a regular surface to a tame stack locally of finite type over an excellent scheme. The proof uses the valuative criterion for proper tame morphisms, which was proved by Bresciani and Vistoli, together with the resolution of singularities for excellent surfaces and the root stack construction. Using Hironaka’s results on the resolution of singularities over fields of characteristic zero, we extend the result to rational maps from a regular scheme of arbitrary dimension to a tame stack locally of finite type over a field of characteristic zero. We also provide a Purity Lemma for higher dimensional tame stacks, generalizing results of Abramovich, Olsson, and Vistoli, which also plays an essential role in the proof.
	\end{abstract}
	\maketitle
	
	\numberwithin{theorem}{section}
	\numberwithin{lemma}{section}
	\numberwithin{corollary}{section}
	\numberwithin{proposition}{section}
	\numberwithin{conjecture}{section}
	\numberwithin{question}{section}
	\numberwithin{remark}{section}
	\numberwithin{definition}{section}
	\numberwithin{example}{section}
	\numberwithin{notation}{section}
	\numberwithin{equation}{section}
    
	\section{Introduction}
	\label{sec:intro}
    Under mild assumptions, a rational map $\varphi:X \dashrightarrow Y$ of schemes has a largest domain of definition $U$ on which $\varphi:U \dashrightarrow Y$ is a morphism. With abuse of notation, we will also denote the morphism $U \to Y$ from the domain of definition, by $\varphi$. A natural question then arises is, given a rational map $\varphi:X \dashrightarrow Y$ with $U$ the domain of definition, whether there exists a morphism $\td{\varphi}: X \to Y$ that extends $\varphi:U \to Y$, i.e., $\td{\varphi}|_U = \varphi$. The answer is not positive in general, but we expect the existence of such an extension after a certain base change $X' \to X$. This problem is called the \emph{resolution of indeterminacy} of rational maps, and our natural hope is to have good control over $X'$ and the extension $X' \to M$. The main result of this note is that we can resolve the indeterminacy of rational maps from a regular surface to a proper tame stack.

    \begin{theorem}\label{thm:main} Let $f:\cM \to \cN$ be a tame, proper morphism of algebraic stacks locally of finite type over an excellent scheme $S$, with $\cM$ separated over $S$. Let $X$ be a regular separated scheme of dimension $2$, locally of finite type over $S$, and $U \subseteq X$ a dense open subscheme of $X$. Suppose we have a commutative diagram
            \begin{center}
                \begin{tikzcd}
                    U \arrow[r] \arrow[d,hook]& \cM \arrow[d,"f"] \\
                    X \arrow[r]& \cN
                \end{tikzcd}
            \end{center}
        Then, there exists a regular algebraic stack $\cX$ of dimension $2$, locally of finite type over $S$, and a birational morphism $\cX \to X$ which is an isomorphism over $U$, and an extension $\cX \to \cM$ making the diagram
            \begin{center}
                \begin{tikzcd}
                   & U \arrow[ld,hook] \arrow[r] \arrow[d,hook]& \cM \arrow[d,"f"] \\
                \cX \arrow[r] \arrow[rru,dashed] & X \arrow[r]& \cN
                \end{tikzcd}
            \end{center}
        commute. Moreover, the morphism $\cX \to X$ factors as
        \begin{align*}
            \cX = \sqrt[\bm{r}]{\bm{D}/X'} \to X' \to X
        \end{align*}
        where $X' \to X$ is a proper birational morphism with a simple normal crossings divisor $X' \setminus U = \cup_{i=1}^N D_i$, and $\sqrt[\bm{r}]{\bm{D}/X'} \to X'$ is a root stack morphism with $\bm{D} = (D_1,\cdots,D_N)$ and the unique minimal $N$-tuple of positive integers $\bm{r}=(r_1,\cdots,r_N)$. Moreover, $X' \to X$ can be factored into a finite sequence of blow-ups at a closed point if the coarse moduli space of $\cM \times_{\cN} X$ is projective.
        \end{theorem}

    Theorem \ref{thm:main} and its proof were motivated by the following results on the resolution of indeterminacy of rational maps from regular curves. When $X$ is regular and of dimension $1$, we have the well-known \emph{valuative criterion for properness}, which says that, if a morphism $f:M \to N$ of schemes is proper, then given a commutative solid diagram
    \begin{center}
        \begin{tikzcd}
            \Spec K \arrow[r] \arrow[d,hook]& M \arrow[d] \\
            \Spec R \arrow[r] \ar[ru,dashed]& N
        \end{tikzcd}
    \end{center}
    where $R$ is a DVR with quotient field $K$, there exists a unique morphism $\Spec R \to M$ making the diagram commute. 
    
    Using the valuative criterion for properness together with the \emph{spreading out} argument, we can resolve the indeterminacy of rational maps from a regular curve as follows. Let $X \dashrightarrow M$ be a rational map from a complete, irreducible, nonsingular curve $X$ to a proper scheme $M$ where both $X$ and $M$ are over an algebraically closed field $k$. If $U$ is the domain of definition, $X \setminus U$ consists of finitely many closed points, say $p_1,\cdots,p_k$. For each $p_i$, we have a morphism $\Spec \cO_{X,p_i} \to M$ extending $\Spec \cO_{X,\eta} = \Spec K \to M$ by the valuative criterion where $\eta$ is the generic point of $X$. After spreading out each $\Spec \cO_{X,p_i} \to M$ to a morphism $V_i \to M$ from an open affine neighborhood $V_i$ of $p_i$ (for example, see \cite[Proposition 8.14.2]{EGAIV-3} or \cite[Tag \spref{01ZC}]{stacks-project}), we can glue these morphisms to a morphism $X \to M$ extending the given $U \to M$. 
    
	However, for rational maps to an algebraic space or algebraic stack, this approach is not sufficient. Indeed, the valuative criterion for proper morphisms of algebraic stacks (for example, \cite[Theorem 3.8.7]{AlperStacks}, \cite[Tag \spref{0CLZ}]{stacks-project}) says that for a proper morphism of algebraic stacks $\cM \to \cN$, we only have a lifting $\Spec R' \to \cM$ of the composition $\Spec R' \to \Spec R \to \cN$ extending the morphism $\Spec K' \to \Spec K \to \cM$ where $R \to R'$ is an extension of DVRs and $K' = \Frac R'$. Moreover, such an extension of DVRs is necessary as remarked in \cite[Example 3.8.16]{AlperStacks}. Even for proper morphisms of algebraic spaces, we need an extension of DVRs $R' \to R$ \cite[Tag \spref{0A40}]{stacks-project}.
   
    The construction in the valuative criterion for algebraic stacks is not satisfactory, since $\Spec R' \to \Spec R$ can never be a birational morphism unless $K = K'$ and the extension $R \to R'$ typically induces a non-trivial extension of residue fields as remarked in \cite[\S 1]{BrescianiVistoliValCrit}. In \cite{BrescianiVistoliValCrit}, a more suitable extension in this context, avoiding an extension of DVRs, was provided by considering the notion of \emph{tame stacks}. In the sense of \cite{AbramovichOlssonVistoliTame}, tame stacks are algebraic stacks whose inertia stack is finite and has linearly reductive geometric fibers. More generally, a morphism $f: \cM \to \cN$ of algebraic stacks is \emph{tame} if the relative inertia stack $\cI_{\cM/\cN} \to \cM$ is finite and has linearly reductive geometric fibers (see Section \ref{sec:notations}). They showed in \cite[Theorem 3.1]{BrescianiVistoliValCrit} that there exists a representable morphism $\sqrt[n]{\Spec R} \to \cM$ which makes the diagram commute, with $n$ minimal, where $\sqrt[n]{\Spec R}$ is a \emph{root stack} along the divisor defined by a uniformizer of $R$, which will be defined later in Section \ref{sec:root-stacks}. The advantage of this construction is that the root stack morphism $\sqrt[n]{\Spec R} \to \Spec R$ is a coarse moduli space and, in particular, it is a proper birational morphism. Here, we say that a morphism $\cM \to \cN$ is \emph{representable} if for any morphism $T \to \cN$ from a scheme $T$, the fiber product $\cM \times_\cN T$ is an algebraic space. 

    In particular, if $f:M \to N$ is a proper morphism of algebraic spaces from a separated algebraic space $M$, and $X \dashrightarrow M$ is a rational map over $N$ from a regular curve $X$ with indeterminacy locus $\{p_1,\cdots,p_k\}$, we can extend it to a morphism $X \to M$ as in the case where $f$ is a morphism of schemes. Indeed, the valuative criterion for proper tame morphisms \cite[Theorem 3.1]{BrescianiVistoliValCrit} gives an extension $\sqrt[n]{\Spec \cO_{X,p_i}} \to M$ for some $n$, but $n = 1$ since this map is representable, and thus we get a morphism $\Spec \cO_{X,p_i} \to M$ for each $i=1,\cdots,k$. Then, by a similar characterization for morphisms locally of finite presentation between algebraic spaces (for example, see \cite[Tag \spref{04AK}]{stacks-project}), we can spread out each of these morphisms to a morphism from an open neighborhood of $p_i$ and glue them to get a morphism $X \to M$. 

     Moreover, using the valuative criterion for proper tame morphisms and the spreading out argument based on Lemma \ref{lem:loc-fp}, we can also resolve the indeterminacy of a rational map from a regular curve to a proper tame stack, after taking a root stack of the given regular curve. However, if $X$ is a surface, we cannot expect such a resolution even if $X$ is regular and the target is a scheme. For example, if $\td{X}$ is a blow-up of $X$ at a closed point, $\td{X}$ is birational to $X$ so that we have a rational map $X \dashrightarrow \td{X}$, but we can never obtain a morphism $X \to \td{X}$. However, in general, given a rational map $X \dashrightarrow M$ where $M$ is proper, we can construct an extension $X' \to M$, where $X'$ is equipped with a proper birational map $X' \to X$ by considering the graph of the given rational map. If we further assume that $M$ is projective, then the birational map $X' \to X$ is also projective and thus we can even describe $X'$ as a blow-up of $X$. 

    Motivated by these observations, we resolved the indeterminacy of a rational map $X \dashrightarrow \cM$ from a regular surface $X$ to a proper tame stack $\cM$ by first resolving the singularities of indeterminacy locus and then taking a root stack $\cX \to X'$. Accordingly, base changes to $X'$ and then to $\cX$ in Theorem \ref{thm:main} are necessary if we consider rational maps from $X$ to its blow-up as above, and a root stack \cite[Example 3.8.17]{AlperStacks}, respectively. The tameness assumption is necessary as well by \cite[Example 3.3]{BrescianiVistoliValCrit}, if we precompose the rational map $X \dashrightarrow \cM$ with an inclusion from a nonsingular curve on $X$ which is not contained in $U$. 

    As a special case of Theorem \ref{thm:main}, we can consider the case where $f:\cM \to \cN$ is a coarse moduli space. Indeed, if we replace $f:\cM \to \cN$ by a coarse moduli space, we can also prove a similar result for $X$ a threefold.
    \begin{theorem}\label{thm:main-threefold} Let $\cM$ be a separated tame stack locally of finite type over an excellent scheme $S$, and $f:\cM \to M$ a coarse moduli space. Let $X$ be a regular separated scheme of dimension $3$, locally of finite type over $S$, and $U \subseteq X$ a dense open subscheme of $X$. Suppose we have a commutative diagram
            \begin{center}
                \begin{tikzcd}
                    U \arrow[r] \arrow[d,hook]& \cM \arrow[d,"f"] \\
                    X \arrow[r]& M
                \end{tikzcd}
            \end{center}
        Then, there exists a regular algebraic stack $\cX$ of dimension $3$, locally of finite type over $S$, and a birational morphism $\cX \to X$ which is an isomorphism over $U$, and an extension $\cX \to \cM$ making the diagram
            \begin{center}
                \begin{tikzcd}
                   & U \arrow[ld,hook] \arrow[r] \arrow[d,hook]& \cM \arrow[d,"f"] \\
                \cX \arrow[r] \arrow[rru,dashed] & X \arrow[r]& M
                \end{tikzcd}
            \end{center}
        commute. Moreover, the morphism $\cX \to X$ factors as
        \begin{align*}
            \cX = \sqrt[\bm{r}]{\bm{D}/X'} \to X' \to X
        \end{align*}
        where $X' \to X$ is a projective birational morphism with a simple normal crossings divisor $X' \setminus U = \cup_{i=1}^N D_i$, and $\sqrt[\bm{r}]{\bm{D}/X'} \to X'$ is a root stack morphism with $\bm{D} = (D_1,\cdots,D_N)$ and the unique minimal $N$-tuple of positive integers $\bm{r} = (r_1,\cdots,r_N)$. Moreover, $X' \to X$ can be factored into a finite sequence of blow-ups at a closed point.
        \end{theorem}
    
    Moreover, the conditions on $X$ in Theorem \ref{thm:main} are assumed to apply the resolution of singularities for excellent surfaces \cite{CossartJannsenSaito}. However, by Hironaka \cite{Hironaka}, we have the resolution of singularities for schemes of arbitrary dimension over a field of characteristic 0. Therefore, we obtain the following result for rational maps from regular schemes of arbitrary dimension.

     \begin{theorem}\label{thm:main-higher-dim} Let $f:\cM \to \cN$ be a tame, proper morphism of algebraic stacks locally of finite type over a field $k$ of characteristic $0$, with $\cM$ separated over $k$. Let $X$ be a regular separated scheme of dimension $n \ge 2$, locally of finite type over $k$, and $U \subseteq X$ a dense open subscheme of $X$. Suppose we have a commutative diagram
            \begin{center}
                \begin{tikzcd}
                    U \arrow[r] \arrow[d,hook]& \cM \arrow[d,"f"] \\
                    X \arrow[r]& \cN
                \end{tikzcd}
            \end{center}
        Then, there exists a regular algebraic stack $\cX$ of dimension $n$, locally of finite type over $k$, and a birational morphism $\cX \to X$ which is an isomorphism over $U$, and an extension $\cX \to \cM$ making the diagram
            \begin{center}
                \begin{tikzcd}
                   & U \arrow[ld,hook] \arrow[r] \arrow[d,hook]& \cM \arrow[d,"f"] \\
                \cX \arrow[r] \arrow[rru,dashed] & X \arrow[r]& \cN
                \end{tikzcd}
            \end{center}
        commute. Moreover, the morphism $\cX \to X$ factors as
        \begin{align*}
            \cX = \sqrt[\bm{r}]{\bm{D}/X'} \to X' \to X
        \end{align*}
        where $X' \to X$ is a proper birational morphism with a simple normal crossings divisor $X' \setminus U = \cup_{i=1}^N D_i$, and $\sqrt[\bm{r}]{\bm{D}/X'} \to X'$ is a root stack morphism with $\bm{D} = (D_1,\cdots,D_N)$ and the unique minimal $N$-tuple of positive integers $\bm{r} = (r_1,\cdots,r_N)$. Moreover, $X' \to X$ can be factored into a finite sequence of blow-ups at a closed point if the coarse moduli space of $\cM \times_{\cN} X$ is projective.
    \end{theorem}

    After posting the first version of this preprint, we became aware of the recent work of Rydh \cite{RydhFlatification}, which independently obtains a similar resolution result for the characteristic zero case. In \cite[Theorem 3.3]{RydhFlatification}, it is shown that for a rational map from a smooth scheme to a Deligne-Mumford stack in characteristic zero, the indeterminacy can be resolved by a functorial sequence of Kummer blow-ups (blow-ups followed by root stacks along the exceptional divisors). While our Theorem \ref{thm:main-higher-dim} provides an explicit construction via root stacks in the same setting, our main result (Theorem \ref{thm:main}) addresses the case for surfaces in arbitrary characteristic (including mixed characteristic), which is distinct from the characteristic zero methods.

    \begin{remark}
        In Theorems \ref{thm:main}, \ref{thm:main-threefold}, and \ref{thm:main-higher-dim}, once we find a stack $\cX = \sqrt[\bm{r}]{\bm{D}/X'}$ with morphisms $\cX \to X$ and $\cX \to \cM$, we can construct a proper birational map $\cX' = \sqrt[\bm{r'}]{\bm{D}/X'} \to \cX$ for any $N$-tuple of positive integers $\bm{r'} = (r_1',\cdots,r_N')$ with $r_i| r_i'$, $i=1,\cdots,N$. Hence, a regular stack $\cX$ filling in the diagram may not be unique. However, we can choose minimal positive integers $r_1,\cdots,r_N$ so that the morphism $\cX \to \cM$ is representable away from the locus where the pullbacks of $D_i$'s to $\cX$ intersect by \cite[Theorem 3.1]{BrescianiVistoliValCrit}. The stack $\cX$ and the morphisms $\cX \to X$ and $\cX \to \cM$ are unique in this sense.

        The morphism $\cX \to \cM$ is not representable in general. If we need a representable map, we may take relative coarse moduli space $\cY \to \cM$ of $\cX \to \cM$ to obtain a representable map (for example, see \cite[Section 3]{AbramovichOlssonVistoliTwisted}). However, the relative coarse moduli space is not regular in general. Therefore, we may have to choose between the representability of the lifting $\cX \to \cM$ and the regularity of the source stack $\cX$, see Example \ref{ex:family-of-elliptic-curves}.
    \end{remark}

    The advantage of our main result, Theorem \ref{thm:main}, is mainly threefold. Firstly, our proof works under sufficiently general assumptions, particularly in arbitrary characteristics, that all schemes and stacks are locally of finite type over an excellent scheme, and algebraic stacks have finite inertia so that their coarse moduli spaces exist. However, such assumptions are achieved in most of the schemes and stacks that appear in practice. Secondly, we have constructed the extension $\cX \to \cM$ with the minimal necessary base changes. Lastly, we have good control over the extension $\cX \to \cM$, in the sense that $\cX$ is also regular and birational to the given regular surface $X$, and the extension is unique up to a unique isomorphism. Moreover, we expect Theorems \ref{thm:main-threefold} and \ref{thm:main-higher-dim} to also be useful in practice, since it is natural to consider rational maps to a tame stack over its coarse moduli space, or to work over a field of characteristic 0. 

    \subsection*{Applications} The main application of our results is an extension of a family of objects of a tame stack $\cM$. For example, the problem of extending torsors has been studied in various settings, and by setting $\cM = \cB G$, we can extend $G$-torsors for a finite, flat, linearly reductive group $G$, see Section \ref{subsec:extension-of-torsors}. More generally, if $\cM$ is a tame moduli stack, our results can be applied to extend a fibration over an open dense subset of a regular variety whose fibers are parametrized by $\cM$, see Section \ref{subsec:fibrations}. Another application is an extension of a morphism defined at the generic point of regular varieties, see Section \ref{subsec:extension-generic}.
    
    \vspace{5pt}
    This note is organized as follows. In Section \ref{sec:root-stacks}, we review the notion of root stacks and prove that root stack construction along a simple normal crossings divisor preserves regularity. In Section \ref{sec:purity-lemma}, we generalize the Purity Lemma which was provided in \cite{AbramovichOlssonVistoliTwisted} to higher dimensional schemes and tame stacks. In Section \ref{sec:proof}, we prove Theorem \ref{thm:main} and provide a sketch of the proofs of Theorems \ref{thm:main-threefold} and \ref{thm:main-higher-dim} which are similar to the proof of Theorem \ref{thm:main}. Lastly, in Section \ref{sec:applications}, we provide applications and an example to illustrate how our results work in practice. 

    \subsection*{Acknowledgements} This paper was written under the supervision of my advisor, Dori Bejleri, at the University of Maryland, College Park. I sincerely thank him for sharing many ideas essential to the results of this work and for his generous guidance throughout. I would also like to thank Philip Engel for posing the question on the higher dimensional Purity Lemma. The research was partially supported by NSF grant DMS-2401483.
	
	\section{Notations and Conventions}
    \label{sec:notations}

    All schemes and stacks we consider are assumed to be locally of finite type over an excellent base scheme. In particular, all schemes we consider are excellent and locally of finite presentation over the base. However, we only require the excellence of schemes for the resolution of singularities, so one can replace the base by a locally noetherian scheme for the argument not using the resolution of singularities. 
    
    If $\cX$ is an algebraic stack, the \emph{inertia stack} $\cI_\cX$ of $\cX$ is the fiber product
    \begin{center}
        \begin{tikzcd}
            \cI_\cX \ar[r] \ar[d] \ar[rd,phantom,"\square"] & \cX \ar[d,"\Delta"]\\
            \cX \ar[r,"\Delta"]& \cX \times \cX
        \end{tikzcd}
    \end{center}
    where $\Delta: \cX \to \cX \times \cX$ is the diagonal morphism. Similarly, the \emph{relative inertia stack} $\cI_{\cX/\cY}$ of a morphism $\cX \to \cY$ is the fiber product $\cI_{\cX/\cY}:=  \cX \times_{\cX \times_\cY \cX} \cX$. We will assume that the algebraic stacks we consider have finite inertia so that coarse moduli spaces exist by Keel-Mori Theorem \cite[Theorem 1.1]{ConradKeelMori}, \cite[Theorem 6.12]{RydhExistenceOfGeometricQuotients}.

    We say that an algebraic stack (resp., Deligne-Mumford stack) $\cX$ is \emph{regular} in the sense of \cite[Definition 3.3.7]{AlperStacks}, that is, it has a smooth (resp., étale) presentation $V \to \cX$ by a regular scheme $V$ (equivalently, all such presentations are given by regular schemes).
    
    \begin{remark}[Regularity can be checked flat locally]\label{rmk:regularity-flat} By \cite[Ch.0, Proposition 17.3.3]{EGAIV-1}, an algebraic stack $\cX$ is regular if it has a flat cover by a regular scheme, that is, there exists a flat, surjective, representable morphism $V \to \cX$ from a regular scheme $V$.
    \end{remark}
    
    We follow the definition of tame stacks and morphisms given in \cite{AbramovichOlssonVistoliTame}. An algebraic stack $\cM$ is said to be \emph{tame} if the inertia stack $\cI_\cM \to \cM$ is finite and has linearly reductive geometric fibers. Equivalently, $\cM$ is either étale or fppf locally on its coarse moduli space a quotient of a scheme by a finite, flat, linearly reductive group scheme \cite[Theorem 3.2]{AbramovichOlssonVistoliTame}.
    
    More generally, a morphism $f:\cM \to \cN$ of algebraic stacks is said to be \emph{tame} if the relative inertia stack $\cI_{\cM/\cN} \to \cM$ is finite and has linearly reductive geometric fibers. Equivalently, the base change $\cZ \times_\cN \cM$, by a morphism $\cZ \to \cN$, is a tame stack, where $\cZ$ is a scheme or a tame stack. Furthermore, $f$ is tame if $\cM$ is tame \cite[Proposition 2.1]{BrescianiVistoliValCrit}. 
    
    By the Local Structure Theorem of Deligne-Mumford stacks \cite[Theorem 4.3.1]{AlperStacks}, above definition for tameness coincides with the usual definition for the definition of tame Deligne-Mumford stack (for example, see \cite[\S 1.3]{AbramovichVistoliCompactifying}, \cite[Definition 4.4.20]{AlperStacks}), which says that a Deligne-Mumford stack $\cM$ is tame if for any geometric point $x:\Spec \Omega \to \cM$, the order of the stabilizer group $G_x$ is prime to the characteristic of the field $\Omega$.
  
	\section{Root Stacks}\label{sec:root-stacks} In this section, we introduce the notion of a root stack. The detailed arguments are given in the literature, for example, in \cite{CadmanRoot}, \cite{AbramovichGrabebrVistoli}, \cite{FantechiMannNironi}. Let $n$ be a positive integer. Consider the quotient stack $[\bA^n / (\bG_m)^n]$ where the action of $(\bG_m)^n$ is given by multiplication on the coordinates. Note that the objects of $[\bA^n/(\bG_m)^n]$ are $n$ line bundles $(\cL_1,\cdots,\cL_n)$ with $n$ global sections $(s_1,\cdots,s_n)$ and such data are equivalent to $n$ effective Cartier divisors $(\cD_1,\cdots,\cD_n)$.
    
    Given $n$ effective Cartier divisors $\bm{D} = (\cD_1,\cdots, \cD_n)$ on a stack $\cX$, and an $n$-tuple of positive integers $\bm{r} = (r_1,\cdots,r_n)$, the root stack $\sqrt[\bm{r}]{\bm{D}/\cX}$ is defined by a Cartesian square
    \begin{center}
        \begin{tikzcd}
            \sqrt[\bm{r}]{\bm{D}/\cX} \arrow[r] \arrow[d] \ar[rd,phantom,"\square"] & \left[ \bA^n/(\bG_m)^n  \right] \arrow[d,"\bm{r}"]\\
            \cX \arrow[r,"\bm{D}"] & \left[ \bA^n/(\bG_m)^n \right]
        \end{tikzcd}
    \end{center}
    where the vertical map $\bm{r}:[\bA^n/(\bG_m)^n] \to [\bA^n/(\bG_m)^n]$ is $(\cD_1',\cdots,\cD_n') \mapsto (r_1 \cD_1',\cdots, r_n \cD_n')$. Then the canonical morphism $\sqrt[\bm{r}]{\bm{D}/\cX} \to \cX$ is an isomorphism over $\cX \setminus \cup_i \cD_i$ and if $\cX$ is a scheme, it is a coarse moduli space.
    
   \begin{remark}\label{rmk:root-stack-local-description}
        When $\cX =\Spec R$ is an affine scheme and each $\cD_i$ is defined by a global section $x_i \in R$, we have a concrete description of the root stack as a quotient stack
    \begin{align*}
        \sqrt[\mathbf{r}]{\bm{D}/\cX} = [(\Spec R[t_1,\cdots,t_n]/(t_1^{r_1}-x_1,\cdots,t_n^{r_n}-x_n))/(\mu_{r_1} \times \cdots \times \mu_{r_n})]
    \end{align*}
    where each $\mu_{r_i}$ acts by $t_i \mapsto \zeta_i \cdot t_i$. So we have an étale presentation of $\sqrt[\mathbf{r}]{\bm{D}/\cX}$ if the residue characteristics of $R$ do not divide all $r_i$'s. Thus, under this assumption, $\sqrt[\mathbf{r}]{\bm{D}/\cX}$ is a Deligne-Mumford stack.

    Meanwhile, another description of $\sqrt[\mathbf{r}]{\bm{D}/\cX}$ as a quotient stack was given in \cite{AbramovichGrabebrVistoli} as follows.
    \begin{align*}
        \sqrt[\mathbf{r}]{\bm{D}/\cX} = [(\Spec R[u_1,\cdots,u_n,t_1,t_1^{-1}\cdots,t_n,t_n^{-1}]/(u_1^{r_1}t_1-x_1,\cdots,u_n^{r_n}t_n-x_n))/\bG_m^n]
    \end{align*}
    where $\bG_m^n$ acts by $(\alpha_1,\cdots,\alpha_n)\cdot (u_1,t_1,\cdots,u_n,t_n) = (\alpha_1 u_1,\alpha_1^{-r_1} t_1,\cdots,\alpha_n u_n,\alpha_n^{-r_n}t_n)$. Since $\bG_m^n$ is smooth over an arbitrary base, we have a smooth presentation of $\sqrt[\mathbf{r}]{\bm{D}/\cX}$ so that $\sqrt[\mathbf{r}]{\bm{D}/\cX}$ is always an algebraic stack.
   \end{remark}

   \begin{remark}\label{rmk:root-stacks-are-tame}
   Consider the case where $\cX$ is an arbitrary scheme. By the local description of root stacks (Remark \ref{rmk:root-stack-local-description}), the root stack $\sqrt[\bm{r}]{\bm{D}/\cX}$ is always an algebraic stack, and if $r_i$'s are not divided by all residue characteristics of the base scheme, the root stack $\sqrt[\bm{r}]{\bm{D}/\cX}$ is Deligne-Mumford. Moreover, by \cite[Theorem 3.2]{AbramovichOlssonVistoliTame}, it is tame.

    Meanwhile, note that the group scheme $\mu_r$ is still fppf over the base even if $r$ is divided by some residue characteristic. Therefore, by \cite[Theorem 3.2]{AbramovichOlssonVistoliTame}, the root stack $\sqrt[\bm{r}]{\bm{D}/\cX}$ is always tame.
   \end{remark}

    Now we provide a series of results that the root stack construction preserves the regularity when the divisor is a simple normal crossings divisor.

    \begin{lemma}\cite[Exercise 5.5]{AM}\label{lem:commalg}
        Let $R \subset S$ be an integral ring extension. Then $r \in R$ is a unit in $R$ if and only if it is a unit in $S$.
    \end{lemma}
    \begin{proof}
        Let $r \in R$ be a unit and $rs = 1$ with $s \in R$. Then $rs = 1$ in $S$ and hence $r$ is a unit in $S$. Conversely, let $r \in R$ be a unit in $S$ and $rs=1$ with $s \in S$. $s$ is integral over $R$ and there is an integral relation
        \begin{align*}
            s^k + a_{k-1}s^{k-1} + \cdots + a_1 s + a_0 = 0
        \end{align*}
        with $a_i \in R$. Multiplying $r^{k-1}$, we have
        \begin{align*}
            s = -( a_{k-1} + a_{k-2}r + \cdots a_1 r^{k-2} + a_0 r^{k-1}) \in R
        \end{align*}
    \end{proof}

    \begin{lemma}\label{lem:root-stack-cover-regular}
    Let $R$ be a regular local ring of dimension $n$ and $x_1,\cdots,x_n$ a regular sequence of $R$ generating the maximal ideal $\fkm$. Let $r_1,\cdots,r_n$ be positive integers. Then the ring
    \begin{align*}
        R[t_1,\cdots,t_n]/(t_1^{r_1}-x_1,\cdots,t_n^{r_n}-x_n)
    \end{align*}
    is a regular local ring of dimension $n$.
    \end{lemma}
    \begin{proof}   
    Let $S = R[t_1,\cdots,t_n]/(t_1^{r_1}-x_1,\cdots,t_n^{r_n}-x_n)$ and denote the residue field of $R$ by $k$. Since $S$ is finite over $R$, it is an integral extension of $R$. Thus, we have $\dim S = \dim R = n$. Let $\fkm_S = (t_1,\cdots,t_n)$. Then $S/\fkm_S = R/\fkm =k$ is a field so that $\fkm_S$ is a maximal ideal of $S$. We claim that $\fkm_S$ is the unique maximal ideal of $S$. Suppose $\fkm' \subseteq S$ is another maximal ideal of $S$. Let $\fkp = \fkm' \cap R$. If $x_i \in \fkp$ for all $i$, then $t_i \in \fkm'$ for all $i$. So we have $\fkm' \supseteq \fkm_S$ and thus $\fkm' = \fkm_S$. On the other hand, suppose $x_i \notin \fkp$ for some $i$, without loss of generality, say $x_1 \notin \fkp$. In particular, we also have $x_1 \notin \fkm'$. We have an inclusion $R/\fkp \hookrightarrow S/\fkm'$ which is an integral ring extension. Since $x_1 \in \fkm$, the image of $x_1$ in $R/\fkp$ is a non-unit. But $S/\fkm'$ is a field so the image of $x_1$ in $S/\fkm'$ is a unit and it is a contradiction by Lemma \ref{lem:commalg}. 
    
    Therefore, we have $\fkm' = \fkm_S$. So $\fkm_S$ is the unique maximal ideal of $S$. Since $\fkm_S$ is generated by $n$ elements, $\fkm_S/\fkm_S^2$ can be generated by $n$ elements over $k = S/\fkm_S$. Then we have
    \begin{align*}
        n = \dim S \le \dim_k \fkm_S/\fkm_S^2 \le n.
    \end{align*}
    It follows that $\dim S = \dim_k \fkm_S/\fkm_S^2 = n$. $S$ is a regular local ring of dimension $n$.
    \end{proof}

   \begin{corollary}\label{cor:local-root-stack-regular}
        Let $R$ be a regular local ring of dimension $n$ and $x_1,\cdots,x_n$ a regular sequence of $R$ generating the maximal ideal $\fkm$. Let $D_1,\cdots,D_n$ be effective divisors defined by $x_1,\cdots,x_n$, respectively, and $r_1,\cdots,r_n$ positive integers. Then the root stack
        \begin{align*}
            \sqrt[(r_1,\cdots,r_n)]{(D_1,\cdots,D_n)/\Spec R}
        \end{align*}
        is regular.
    \end{corollary}
    \begin{proof}
        If residue characteristics of $R$ do not divide all $r_i$'s, $\mu_{r_1}\times \cdots \times \mu_{r_n}$ is étale, and thus an étale cover
        \begin{align*}
            \Spec R[t_1,\cdots,t_n]/(t_1^{r_1}-x_1,\cdots,t_n^{r_n}-x_n) \to \sqrt[(r_1,\cdots,r_n)]{(D_1,\cdots,D_n)/\Spec R}
        \end{align*}
        gives an étale presentation by a regular scheme. However, even if a residue characteristic of $R$ divides some $r_i$, the above cover is still a flat cover, and thus the root stack $\sqrt[(r_1,\cdots,r_n)]{(D_1,\cdots,D_n)/\Spec R}$ is still regular by Remark \ref{rmk:regularity-flat}.
    \end{proof}

    \begin{proposition}\label{prop:regularity-of-root-stack}
        Let $X$ be a regular scheme of dimension $n$ and $\bm{D} = (D_1,\cdots,D_N)$ be $N$ effective Cartier divisors of $X$ such that $\cup_{i=1}^N D_i$ is a simple normal crossings divisor. Let $\bm{r}$ be an $N$-tuple of positive integers. Then the root stack $\cX = \sqrt[\bm{r}]{\bm{D}/X}$ is a regular algebraic stack of dimension $n$, where the union of reduced pullbacks of $D_1,\cdots,D_N$ to $\cX$ is a simple normal crossings divisor. 
    \end{proposition}
    \begin{proof}
        Regularity can be checked étale or flat locally, so it follows immediately from Corollary \ref{cor:local-root-stack-regular}.
    \end{proof}
  
    \section{The Purity Lemma}\label{sec:purity-lemma} In \cite[\S 2]{AbramovichVistoli}, Abramovich and Vistoli proved that the indeterminacy of a rational map from an $S_2$ scheme $X$ of dimension 2 to a separated Deligne-Mumford stack $\cM$ with coarse moduli space $M$ can be resolved if the indeterminacy locus $P$ consists of finitely many closed points and local fundamental groups of $U = X \setminus P$ around the points of $P$ are trivial. Later in \cite[Lemma 4.6]{AbramovichOlssonVistoliTwisted}, Abromovich, Olsson, and Vistoli extended this result to the case where $\cM$ is a separated tame stack with further assumption that the local Picard groups of $U$ around the points of $P$ are torsion free. These results are called the \emph{Purity Lemma} and they say that \emph{the indeterminacy of a rational map happens in codimension 1 in such cases}.
    
    In this section, we provide several results generalizing the Purity Lemma given in \cite{AbramovichOlssonVistoliTwisted}. In fact, the proof of the Purity Lemma in \cite{AbramovichOlssonVistoliTwisted} works for $X$ of higher dimension, not only for dimension 2, so we restated the result in Lemma \ref{lem:purity-higher-dimensional-I} without proof. In Lemma \ref{lem:purity-higher-dimensional-II}, we extend Lemma \ref{lem:purity-higher-dimensional-I} to the case where the indeterminacy locus is of lower codimension but still greater than equal to $2$. Lastly, in Lemma \ref{lem:purity-tame-stack}, we extend the result to the case where $X$ is replaced by a regular tame stack.

    \begin{lemma}[Purity Lemma for Higher Dimensions I]\label{lem:purity-higher-dimensional-I}
    Let $\cM$ be a separated tame stack with coarse moduli space $M$. Let $X$ be a locally noetherian separated scheme of dimension $n \ge 2$ satisfying Serre’s condition $S_2$. Let $P \subset X$ be a finite subset consisting of closed points, $U = X \setminus P$. Assume that the local fundamental groups of $U$ around the points of $P$ are trivial and that the local Picard groups of $U$ around points of $P$ are torsion free, i.e., $\pi_1^{\text{ét}}(\Spec \cO_{X,p}^{\sh} \setminus p) =0$, $\Pic(\Spec \cO_{X,p}^{\sh} \setminus p)$ is torsion-free for $p \in P$. Let $f : X \to M$ be a morphism and suppose there is a lifting $f_U : U \to M$:
            \begin{center}
                \begin{tikzcd}
                    & & \cM \arrow[d]\\ 
                    U \arrow[r,hook] \arrow[rru,"\td{f}_{U}"] & X  \arrow[r,"f"] & M
                \end{tikzcd}
            \end{center}
            Then the lifting extends to $X$:
            \begin{center}
                \begin{tikzcd}
                    & & \cM \arrow[d]\\ 
                    U \arrow[r,hook] \arrow[rru,"\td{f}_U"] & X \arrow[r,"f"] \arrow[ru, dashed,"\td{f}"] & M
                \end{tikzcd}
            \end{center}
        The lifting $\td{f}$ is unique up to a unique isomorphism.
\end{lemma}
\begin{proof}
   The argument used in the proof of the Purity Lemma \cite[Lemma 4.6]{AbramovichOlssonVistoliTwisted} still applies.
\end{proof}

\begin{lemma}[Purity Lemma for Higher Dimensions II]\label{lem:purity-higher-dimensional-II}
      Let $\cM$ be a separated tame stack, locally of finite presentation, with coarse moduli space $M$. Let $X$ be a locally noetherian separated scheme of dimension $n \ge 2$ satisfying Serre's condition $S_2$ and $U \subseteq X$ an open dense subscheme. Suppose that $Z:=X \setminus U= \cup_{i=1}^k{Z_i}$ where $Z_i$'s are irreducible components of $Z$ with $\codim_X Z_i \ge 2$ for all $i=1,\cdots,k$. Assume further that the local fundamental groups of $U$ around the (possibly non-closed) points of $Z$ are trivial and the local Picard groups of $U$ around the (possibly non-closed) points of $Z$ are torsion free, i.e., $\pi_1^{\text{ét}}(\Spec \cO_{X,p}^{\sh} \setminus p) =0$, $\Pic(\Spec \cO_{X,p}^{\sh} \setminus p)$ is torsion free for all $p \in Z$. Let $f : X \to M$ be a morphism and suppose there is a lifting $f_U : U \to M$:
            \begin{center}
                \begin{tikzcd}
                    & & \cM \arrow[d]\\ 
                    U \arrow[r,hook] \arrow[rru,"\td{f}_{U}"] & X  \arrow[r,"f"] & M
                \end{tikzcd}
            \end{center}
            Then the lifting extends to $X$:
            \begin{center}
                \begin{tikzcd}
                    & & \cM \arrow[d]\\ 
                    U \arrow[r,hook] \arrow[rru,"\td{f}_U"] & X \arrow[r,"f"] \arrow[ru, dashed,"\td{f}"] & M
                \end{tikzcd}
            \end{center}
        The lifting $\td{f}$ is unique up to a unique isomorphism.
\end{lemma}
\begin{proof}
    We may assume that $Z$ is reduced by giving the reduced induced structure to $Z$. We prove by induction on the maximal dimension $N$ of irreducible components of $Z$.
    
    The case $N = 0$ follows immediately from Lemma \ref{lem:purity-higher-dimensional-I}. Suppose $N>1$ and assume that the theorem holds for the maximal dimension of irreducible components of such $Z$ less than $N$. Let $i_1,\cdots,i_l$ be indices such that $\dim Z_{i_j} = N$, and $\eta_j$ the generic point of $Z_{i_j}$. Then $\cO_{X,\eta_j}$ is a regular local ring of dimension $\codim_X Z_{i_j} = \dim X - N$. Moreover, $\eta_j$ is the unique closed point of $\Spec \cO_{X,\eta_j}$, and we have the following diagram induced by the given morphisms.
    \begin{center}
        \begin{tikzcd}
            \Spec \cO_{X,\eta_j} \setminus \eta_j \ar[r] \ar[d,hook] & \cM \ar[d] \\
            \Spec \cO_{X,\eta_j} \ar[r]& M
        \end{tikzcd}
    \end{center}
    Since $\Spec \cO_{X,\eta_j}$ is a regular scheme of dimension $\codim_X Z_{i_j} \ge 2$, Lemma \ref{lem:purity-higher-dimensional-I} gives a morphism $\Spec \cO_{X,\eta_j} \to \cM$ fitting into the diagram by the assumption on the local fundamental groups and local Picard groups. Since $\cM$ is locally of finite presentation, there exists an open neighborhood $V_j \subseteq X$ of $\eta_j$ with a morphism $V_j \to \cM$ which represents the morphism $\Spec \cO_{X,\eta_j} = \lim_{\eta_i \in W} W \to \cM$ where $W$ runs over all open affine neighborhoods of $\eta_j$. Let $V = U \cup (\cup_j V_j)$. The morphisms $U \to \cM$ and $V_j \to \cM$ agree on the overlaps so that they glue to a morphism $V \to \cM$. Meanwhile, since $V$ is dense in $X$, $V \cap Z$ is dense in $Z$. So $Z \setminus V$ consists of irreducible components of dimension strictly less than $N = \max_i \dim Z_i$. By induction, $V \to \cM$ extends to a morphism $X \to \cM$.
\end{proof}

    \begin{lemma}\label{lem:purity-regular-local-fund-pic}
          Let $X$ be a regular, locally noetherian, separated scheme of dimension $n \ge 2$ and $U \subseteq X$ an open dense subscheme. Suppose that $X \setminus U= \cup_{i=1}^k{Z_i}$ where, for all $i=1,\cdots,k$, $Z_i$ is of codimension $\ge 2$. Then the local fundamental groups and the local Picard groups of $U$ around the (possibly non-closed) points of $Z$ are trivial i.e., $\pi_1^{\text{ét}}(\Spec \cO_{X,p}^{\sh} \setminus p) =0$, $\Pic(\Spec \cO_{X,p}^{\sh} \setminus p) =0 $ for all $p \in Z$. 
    \end{lemma}
    \begin{proof}
        Let $p \in Z$ be an arbitrary point and $R = \cO_{X,p}^{\sh}$. Then $R$ is a regular local ring of dimension $m := \codim_X \ov{\{p\}}$, and we have $2 \le m \le n$ by assumptions. 

        If $Q \to \Spec R \setminus p$ is a finite étale cover, Zariski-Nagata Purity (for example, \cite{Nagata-Purity} or \cite[Tag \spref{0BMB}]{stacks-project}) implies that $Q$ extends to a finite étale cover $\td{Q} \to \Spec R$ as $R$ is a regular local ring. Since $R$ is strictly henselian, $\td{Q}$ must be trivial and so is $Q$. Thus, the local fundamental group $\pi_1^{\text{ét}}(\Spec R\setminus p)$ is trivial.
    
        On the other hand, since $R$ is a regular local ring of dimension $m \ge 2$, we have $\Pic (\Spec R \setminus p) \simeq \Pic( \Spec R)$. Meanwhile, since $R$ is a UFD, we have $\Pic (\Spec R) = 0$. Hence, the local Picard group $\Pic ( \Spec R \setminus p)$ is trivial. 
    \end{proof}

    \begin{remark}\label{rmk:purity-regular}
        By Lemma \ref{lem:purity-regular-local-fund-pic}, the assumptions on the local fundamental groups and the local Picard groups in the Purity Lemma are satisfied for $X$ regular. Therefore, Lemma \ref{lem:purity-higher-dimensional-I} and \ref{lem:purity-higher-dimensional-II} hold for $X$ regular.
    \end{remark}

    \begin{lemma}[Purity Lemma for Tame Stacks]\label{lem:purity-tame-stack} Let $\cM$ be a separated tame stack, locally of finite presentation, with coarse moduli space $M$. Let $\cX$ be a regular, locally noetherian, separated tame stack of dimension $n \ge 2$ and $\cU \subseteq \cX$ an open substack whose complement consists of finitely many irreducible components of codimension $\ge 2$. Let $f:\cX \to M$ be a morphism and suppose there is a lifting $\td{f}_{\cU}:\cU \to \cM$:
        \begin{center}
            \begin{tikzcd}
                & & \cM \arrow[d]\\ 
                \cU \arrow[r,hook] \arrow[rru,"\td{f}_{\cU}"] & \cX  \arrow[r,"f"] & M
            \end{tikzcd}
        \end{center}
        Then the lifting extends to $\cX$:
        \begin{center}
            \begin{tikzcd}
                & & \cM \arrow[d]\\ 
                 \cU \arrow[r,hook] \arrow[rru,"\td{f}_{\cU}"] & \cX \arrow[r,"f"] \arrow[ru, dashed,"\td{f}"] & M
            \end{tikzcd}
        \end{center}
        Moreover, the lifting $\td{f}$ is unique up to unique isomorphism.
        \end{lemma}
        \begin{proof}
            Let $\cX \to X$ be a coarse moduli space. By \cite[Theorem 3.2]{AbramovichOlssonVistoliTame}, there exists an étale cover $X' \to X$, and a finite, flat, linearly reductive group scheme $G \to X'$ acting on a finite scheme $Y \to X'$ such that $\cX \times_X X' \simeq [Y/G]$. Since $\cX$ is a regular stack of dimension $n$, $Y$ is a regular scheme of dimension $n$, and we have $\cU \times_\cX [Y/G] \simeq [V/G]$ for some open dense subscheme of $V$ with $G$-action on it, whose complement consists of finitely many irreducible components of codimension $\ge 2$. Since the question is local in the étale topology, by replacing $\cX$ and $\cU$ by $[Y/G]$ and $[V/G]$, we may assume that $\cX = [X/G]$ and $\cU = [U/G]$ where $X$ is a regular scheme of dimension $n$, $U \subseteq X$ is an open dense subscheme whose complement consists of finitely many irreducible components of codimension $\ge 2$, and $G$ is a finite, flat, linearly reductive group scheme acting on $X$ such that its restricted action on $U$ commutes with the inclusion $U \hookrightarrow X$.
            
            By Lemma \ref{lem:purity-higher-dimensional-II} and Remark \ref{rmk:purity-regular}, we get a unique extension of the composition $U \to [U/G] \xrightarrow{\td{f}_{\cU}} \cM$ to $\td{f}_X:X \to \cM$. So we have a following commutative diagram
             \begin{center}
                \begin{tikzcd}
                    & & \cM \arrow[d]\\ 
                    \left[U/G \right] \arrow[r,hook] \arrow[rru,"\td{f}_{\cU}"] &  \left[X/G \right]\arrow[r] & M \\
                    U \arrow[u] \arrow[r,hook] & X \arrow[u] \ar[ruu,"\td{f}_X"'] &
                \end{tikzcd}
            \end{center}
            Two compositions $G \times X \simeq X \times_{[X/G]} X \rightrightarrows X \to \cM$ agree when restricted to $G \times U \simeq U \times_{[U/G]} U$ since the morphism $U \to \cM$ descends to $[U/G] \to \cM$. $G \times X$ is a local complete intersection since $X$ is regular, and hence it is $S_2$ (for example, by \cite[Proposition 18.13]{EisenbudCommAlg}). Also, since $\cM$ is separated and has finite inertia, it has finite diagonal. Thus, the two compositions $G \times X \rightrightarrows X \to \cM$ agree by \cite[Lemma 2.3]{DiLorenzoInchiostro}. By flat descent, $X \to \cM$ descends to $\td{f}:[X/G] \to \cM$ and we get the desired result. 
        \end{proof}
        
	\section{Proofs of Theorems \ref{thm:main}, \ref{thm:main-threefold}, and \ref{thm:main-higher-dim}}\label{sec:proof}
        We now prove Theorems \ref{thm:main}, \ref{thm:main-threefold} and \ref{thm:main-higher-dim}. We begin with the proof of Theorem \ref{thm:main}.
        \begin{remark}
            As remarked in Section \ref{sec:intro}, we may take $f: \cM \to \cN$ in Theorem \ref{thm:main} to be a coarse moduli space $\cM \to M$ of a tame stack $\cM$. Indeed, we can reduce the proof to this special case.
        \end{remark}
        \subsection{Proof of Theorem \ref{thm:main} (Reductions)}\label{subsec:proof-reductions}
                Firstly, we reduce to the case where $f:\cM \to \cN$ is a coarse moduli space $\cM \to M$ of a tame stack $\cM$. Indeed, suppose that the theorem holds for the coarse moduli space $\cM \to M$ of a tame algebraic stack $\cM$. For an arbitrary tame proper morphism $f: \cM \to \cN$ of algebraic stacks, consider the fiber product $\cM_X := \cM \times_\cN X$. By the universal property of the coarse moduli space, the projection $\cM_X \to X$ uniquely factors through the coarse moduli space $M_X$ of $\cM_X$. Note that $M_X \to X$ is proper since $\cM_X \to X$ is proper \cite[Theorem 1.1]{ConradKeelMori}. Regarding $U,X$ and $M_X$ as $X$-schemes, we have a rational map $X \dashrightarrow M_X$ over $X$ defined by the composition $U \to \cM_X \to M_X$. Consider the graph $\Gamma \subseteq X \times_X M_X \simeq M_X$ of a rational map $X \dashrightarrow M_X$ with the reduced induced closed subscheme structure. The projection from $\Gamma$ to $X$ is the composition $\Gamma \hookrightarrow M_X \to X$ where $\Gamma \hookrightarrow M_X$ is a closed embedding and $M_X \to X$ is proper. Hence, $\Gamma \to X$ is a proper birational map which is an isomorphism over $U$. By \cite[Theorem 16.4]{CossartJannsenSaito}, there exists a regular surface $\Gamma'$ with a projective birational map $\Gamma' \to \Gamma$ which factors as a finite sequence of blow-ups at a closed point in the singular locus of $\Gamma$. Then the composition $\Gamma' \to \Gamma \to X$ is also a proper birational map which is an isomorphism over $U$. Now we can apply the theorem to the square
        \begin{center}
            \begin{tikzcd}
            U \arrow[r] \arrow[d,hook] & \cM_X \arrow[d]\\
            \Gamma' \arrow[r] & M_X 
            \end{tikzcd}
        \end{center}
        and get the desired result. Note also that if the coarse moduli space $M_X$ is projective, then $\Gamma \to X$ is also projective so that it is isomorphic to a blow-up of $X$. Thus, $X' \to X$ in the theorem can be factored into a finite sequence of blow-ups by the rest of the proof.

        Now we may assume that $X$ is noetherian. Indeed, once we have a result for noetherian case, then for locally noetherian case, we may cover $X$ by noetherian open subschemes $X_i$ and obtain regular stacks $\cX_i$ and morphisms $\cX_i \to \cM$ fitting into the diagrams. According to our construction which will be described in Section \ref{subsec:proof-codim2}, each $\cX_i$ are constructed from $X_i$ by taking a finite sequence of blow-ups $X_i' \to X_i$ (obtained by \cite[Corollary 0.4]{CossartJannsenSaito}, or \cite[Theorem 5.9]{CossartJannsenSaito}) and then taking a root stack $\cX_i \to X_i'$. By \cite[Theorem 16.2]{CossartJannsenSaito}, the sequence $X_i' \to X_i$ is functorial in the sense that it is compatible with automorphisms of $X_i$ (which are compatible with the indeterminacy locus), and Zariski or étale localizations. Moreover, the description of root stacks in Section \ref{sec:root-stacks} shows that the construction of root stacks are also functorial. Therefore, the morphisms $\cX_i \to X_i' \to X_i$ glue to $\cX \to X' \to X$ where $X' \to X$ is a finite sequence of blow-ups and $\cX \to X'$ is a root stack. Moreover, note that $\cX_i$'s are regular and $\cM$ is separated. Since morphisms $\cX_i \to \cM$ agree on the intersections, they uniquely glue to a morphism $\cX \to \cM$ by the following proposition, which was originally given in \cite[Proposition 1.2]{FantechiMannNironi} for Deligne-Mumford stacks, as root stacks are tame (Remark \ref{rmk:root-stacks-are-tame}).
        
        \begin{proposition}\label{prop:separated}
            Let $\cX$, $\cY$ be tame stacks. Assume that $\cX$ is normal and $\cY$ is separated. Let $i: \cU \hookrightarrow \cX$ be a dominant open immersion of the tame stack $\cU$. If $f,g :\cX \to \cY$ are two morphisms of stacks such that there exists a 2-arrow $f \circ i \xRightarrow{\beta} g \circ i$ then there exists a unique 2-arrow $\alpha: f \Rightarrow g$ such that $\alpha \circ \text{id}_i = \beta$.
        \end{proposition}
        \begin{proof}
            By \cite[Theorem 3.2]{AbramovichOlssonVistoliTame}, a tame stack is étale locally a quotient of a scheme by a finite, flat, linearly reductive group scheme. Therefore, the same argument in the proof of \cite[Proposition 1.2]{FantechiMannNironi} still applies.
        \end{proof}
        
        We may also assume that $X$ is irreducible and of dimension $2$. Indeed, since $X$ is regular, it is analytically unibranch. Thus, irreducible components of $X$ are pairwise disjoint. Since $X$ is a disjoint union of irreducible components, we may work on each irreducible component of $X$ to obtain morphisms from regular stacks birational to the irreducible components and glue them all at the end. $X$ may have lower dimensional componenets as well. But there is nothing to do with $0$-dimensional components, and we may deal with $1$-dimensional components by the similar arguments to dimension $2$ case using \cite[Theorem 3.1]{BrescianiVistoliValCrit} as we have seen in Section \ref{sec:intro}.
       
       Now $X \setminus U$ consists of finitely many irreducible components. Note that it cannot have codimension $0$ irreducible components as $X$ is irreducible. So irreducible components of $X \setminus U$ can be either of codimension $1$, that are irreducible curves, or codimension $2$, that are closed points. We will deal with codimension $1$ components and codimension $2$ components separately, first extending the given morphism $U \to \cM$ to codimension $2$ components of $X \setminus U$ and then to codimension $1$ components after taking blow-ups and root stacks.
        
        \subsection{Proof of Theorem \ref{thm:main} (Extension to codimension 2 components)}\label{subsec:proof-codim2} We first resolve the indeterminacy at the codimension $2$ components using the Purity Lemma (\cite[Lemma 4.6]{AbramovichOlssonVistoliTwisted} or Lemma \ref{lem:purity-higher-dimensional-I}). Since $X \setminus U$ has finitely many irreducible components and an open subscheme of a regular scheme is still regular, we can extend the morphism $U \to \cM$ up to the complement of the union of the codimension $1$ irreducible components of $X \setminus U$ in $X$ by applying the Purity Lemma, without taking any blow-ups or root stacks. Thus, we may reduce to the case where $X \setminus U$ is a union of finitely many irreducible curves. 
        
        \subsection{Proof of Theorem \ref{thm:main} (Extension to codimension 1 components)}\label{subsec:proof-codim1}
        Let $X \setminus U$ be a union of finitely many irreducible curves in $X$. By \cite[Corollary 0.4]{CossartJannsenSaito}, there exists a projective surjective morphism $\pi:X' \to X$ which is an isomorphism over $U$, such that $\pi^{-1}(X \setminus U)$, with the reduced induced closed subscheme structure, is a simple normal crossings divisor on $X'$. Furthermore, the morphism $\pi:X' \to X$ is a finite sequence of blow-ups at a closed point. We briefly explain the construction of $\pi:X' \to X$ in our case following the arguments in \cite{CossartJannsenSaito}.
        
        Let $Z = X \setminus U$. By \cite[Theorem 0.3]{CossartJannsenSaito}, there exists a projective surjective morphism $\pi:X' \to X$ which is a finite sequence of blow-ups $X' = X_m \to X_{m-1} \to \cdots \to X_1 \to X_0 = X$ and $Z' = Z_m \to Z_{m-1} \to \cdots \to Z_1 \to Z_0 = Z$ where $X_{i+1} \to X_i$ and $Z_{i+1} \to Z_i$ are the blow-ups along a permissible center $P_i \subseteq (Z_i)_{\text{Sing}}$ in the sense of \cite[Definition 2.1]{CossartJannsenSaito}. So $X_{i+1}$ is regular and $Z_{i+1}$ is the strict transform of $Z_i$ under $X_{i+1} \to X_i$. Moreover, since $Z$ is a finite union of irreducible curves, $P_i$ must be of dimension $0$. So we can construct $\pi:X' \to X$ with each $X_{i+1} \to X_i$ a blow-up at a closed point. Starting with $B_0 = \emptyset$, let $B_{i+1}$ be the complete transform of $B_i$ under the blow-up $X_{i+1} \to X_i$, i.e., the union of the strict transform of $B_i$ in $X_{i+1}$ and the exceptional divisor of the blow-up $X_{i+1} \to X_i$. Then $B' = B_m$ is a simple normal crossings divisor on $X'$ and intersects $Z'$ transversally on $X'$. So $\pi^{-1}(Z) = Z' \cup B'$ is a simple normal crossings divisor on $X'$ as desired.
        
        Identifying $\pi^{-1}(U)$ with $U$, let $X' \setminus U = \pi^{-1}(X \setminus U)_{\text{red}} = \cup_{i=1}^N D_i$ where $D_i$'s are irreducible components of the simple normal crossings divisor $X' \setminus U$ in $X'$. We will set $\cX = \sqrt[\bm{r}]{\bm{D}/X'}$ where $\bm{D} = (D_1,\cdots,D_N)$ for certain $N$-tuple of positive integers $\bm{r}$ and construct a morphism $\cX \to \cM$. We first consider the following series of lemmas. 
        
        \begin{lemma}\label{lem:quotient-stack-limit}
                Let $I$ be a filtered set and $\{T_i,f_{ij}\}$ an inverse system of affine schemes over $I$. Assume that a smooth group scheme $G$ acts on $T_i$'s and transition maps are all $G$-equivariant so that $G$ also acts on $T$, and projection maps $T \to T_i$ are also $G$-equivariant. Then, 
                \begin{align*}
                    [T/ G] = \lim [T_i/G]
                \end{align*}
                in the category of algebraic stacks.
            \end{lemma}
            \begin{proof}
                Since $G$ is smooth, quotient stacks $[T_i/G]$ and $[T/G]$ have smooth presentations by affine schemes, namely, $T_i \to [T_i/G]$ and $T \to [T/G]$. So the quotient stacks we are considering are all algebraic. For each $i$, a $G$-equivariant morphism $T \to T_i$ induces a moprhism $T/G \to T_i/G$ of affine schemes and hence a morphism $[T/G] \to [T_i/G]$ of quotient stacks. Similarly, we have induced  morphisms $\td{f}_{ij}:[T_j/G] \to [T_i/G]$ for $i \le j$. Note that each $\td{f}_{ij}$ is affine and, in particular, representable, since the diagram
                \begin{center}
                    \begin{tikzcd}[column sep=small]
                        T_j \ar[r] \ar[d] \ar[rd,phantom,"\square"]& T_i \ar[d]\\
                        \left[T_j/G\right] \ar[r]& \left[T_i/G\right]
                    \end{tikzcd}
                \end{center}
                is Cartesian and $T_j \to T_i$ is affine. Thus, $\{[T_i/G],\td{f}_{ij}\}$ is an inverse system over $I$ with affine transition maps.

                We claim that $[T/G]$ satisfies the universal property of the limit in the category of algebraic stacks. Suppose that we are given a collection of morphisms $\cX \to [T_i/G]$ from an algebraic stack $\cX$ that are compatible with transition maps. Let $U \to \cX$ be a smooth presentation of $\cX$ and let $i,j \in I$ with $i \ne j$ be given. Since $I$ is filtered, there exists $k \in I$ such that $i \le k$ and $j \le k$. Consider the Cartesian squares
                \begin{center}
                    \begin{tikzcd}[column sep=scriptsize]
                        P_k \arrow[r] \arrow[d] \ar[rd,phantom,"\square"]& T_k\ar[d]\\
                        U \ar[r] & \left[T_k/G\right]
                    \end{tikzcd}
                    \quad
                    \begin{tikzcd}[column sep=scriptsize]
                        P_i \arrow[r] \arrow[d] \ar[rd,phantom,"\square"]& T_i \ar[d]\\
                        U \ar[r]& \left[T_i/G\right]
                    \end{tikzcd}
                \end{center}
                corresponding to the morphisms $U \to \cX \to [T_k/G]$ and $U \to \cX \to [T_i/G]$, respectively. Since the morphisms $\cX \to [T_i/G]$ are compatible with transition maps, we also have a diagram
                \begin{center}
                \begin{tikzcd}[column sep=scriptsize]
                      P_k \ar[r] \ar[d] \ar[rd,phantom,"\square"]& T_k \arrow[r] \arrow[d] \ar[rd,phantom,"\square"] & T_i \ar[d]\\
                      U \ar[r]& \left[T_k/G\right] \ar[r] & \left[T_i/G\right]
                \end{tikzcd}  
                \end{center}
                where the squares are Cartesian. Therefore, $P_k \to U$ and $P_i \to U$ are isomorphic principal $G$-bundles and the isomorphism between them commutes with the $G$-equivariant maps $P_k \to T_k$ and $P_i \to T_i$. Applying the same argument on morphisms $U \to \cX \to [T_k/G]$ and $U \to \cX \to [T_j/G]$, we may assume that all $\cX \to [T_i/G]$ correspond to the same principal $G$-bundle $P \to U$ with the $G$-equivariant maps $P \to T_i$ that are compatible with transition maps. So we also obtain a $G$-equivariant map $P \to T$ and thus a morphism $U \to [T/G]$ commuting with the given morphisms. Note that this morphism is uniquely determined by construction up to an isomorphism.

                Moreover, it uniquely factors through $\cX$ by descent. Indeed, consider the diagram
                \begin{center}
                    \begin{tikzcd}[column sep=small]
                        U \times_\cX U \arrow[r,shift left] \ar[r,shift right] & U \arrow[r] \ar[d]& \cX \ar[d]\\
                        &  \left[ T/G \right] \ar[r] & \left[ T_i/G \right]
                    \end{tikzcd}
                \end{center}
                which corresponds to the diagram
                \begin{center}
                    \begin{tikzcd}[column sep=small]
                       V \ar[r] &P' \ar[r,shift left] \ar[r,shift right] \ar[d] \ar[rd,phantom,"\square"] & P \ar[r]\ar[d] & T \ar[r]& T_i \\
                       & U \times_\cX U \ar[r,shift left] \ar[r,shift right] & U &&
                    \end{tikzcd}
                \end{center}
                where vertical maps are principal $G$-bundles and $V \to P'$ is an étale presentation of the algebraic space $P'$. Since $U \to [T_i/G]$ factors through $\cX$, two compositions of the top horizontal maps coincide for all $i$. By the universal property of the limit, these two maps uniquely factors through $T$. That is, two compositions $V \to P' \rightrightarrows P \to T$ are equal so that $U \times_\cX U \rightrightarrows U \to [T/G]$ are the same as well. By descent, $U \to [T/G]$ uniquely factors through $\cX$. 

                Therefore, $[T/G]$ satisfies the universal property of the limit in the category of algebraic stacks and we get the desired result.
        \end{proof}

        \begin{remark}\label{rmk:root-stack-as-a-limit}
            Let $X=\Spec R$ with $R$ a regular local ring. Let $\bm{D} = (D_1,\cdots,D_N)$ be effective Cartier divisors on $X$ such that $\cap_{i=1}^N D_i$ is nonempty. Pick $x \in \cap_{i=1}^N D_i$. Let $Y = \Spec \cO_{X,x}$ and let $\bm{D'}=(D_1',\cdots,D_N')$ be the effective Cartier divisors on $Y$ which are pullbacks of $\bm{D}=(D_1,\cdots,D_N)$ to $Y$. Then Lemma \ref{lem:quotient-stack-limit} implies that
            \begin{align*}
                \sqrt[\bm{r}]{\bm{D'}/Y} = \lim_{x \in V} \sqrt[\bm{r}]{\bm{D}_V/V}
            \end{align*}
            where $\bm{D}_V = (D_1 \cap V,\cdots,D_N \cap V)$ and the limit runs over all open neighborhoods $V\subseteq X$ of $x$, by using the smooth presentation of root stacks we have seen in Remark \ref{rmk:root-stack-local-description}.
        \end{remark}
                
        It is known that if $\cM$ is an algebraic stack locally of finite presentation over a scheme $S$ and $T$ is an affine scheme over $S$ which is a limit $T=\lim T_i$ of a filtered inverse system of affine schemes $T_i$ over $S$, then the natural map
        \begin{align*}
            \colim \Mor_S(T_i,\cM) \to \Mor_S (T,\cM)
        \end{align*}
        is an isomorphism (For example, see \cite[Tag \spref{0CMX}]{stacks-project}). Using Lemma \ref{lem:quotient-stack-limit}, we can slightly improve this as follows.

        \begin{lemma}\label{lem:loc-fp} Let $\cM$ be an algebraic stack locally of finite presentation over a scheme $S$ and $G$ a smooth affine group scheme over $S$. Let $I$ be a filtered set and $\{T_i,f_{ij}\}$ an inverse system of affine schemes over $S$ with $T = \lim T_i$, where $T_i$'s are equipped with $G$-actions and transition maps are $G$-equivariant. Then there exists a natural map
        \begin{align*}
            \colim \Mor_S([T_i/G],\cM) \to \Mor_S ([T/G], \cM)
        \end{align*}
        which is an isomorphism.
        \end{lemma}
        \begin{proof}
            As we have seen in the proof of Lemma \ref{lem:quotient-stack-limit}, $f_{ij}$'s induce morphisms $\td{f}_{ij}:[T_j/G] \to [T_i/G]$ for $i \le j$ so that $\{ [T_i/G],\td{f}_{ij} \}$ is a filtered inverse system over $I$ and $[T/G] = \lim [T_i/G]$ in the category of algebraic stacks. So we have a natural map as above. Let $\pi_i: [T/G] \to [T_i/G]$ be projections. Then the above map sends a representative $\phi_i:[T_i/G] \to \cM$ of an element of $\colim \Mor_S([T_i/G],\cM)$ to the composition $\phi_i \pi_i$. Note that the composition doesn't depend on the choice of a representative. Indeed, if $\phi_j:[T_i/G] \to \cM$ is another representative, there exists $k \in I$ such that $i \le k$ and $j \le k$ since $I$ is filtered. Then $\phi_i \pi_i = \phi_i  \td{f}_{ik}  \pi_k = \phi_k\pi_k$ and similarly, $\phi_j \pi_j = \phi_k \pi_k$. So we have $\phi_i\pi_i = \phi_j\pi_j$ and the natural map $\colim \Mor_S([T_i/G],\cM) \to \Mor_S ([T/G], \cM)$ is well-defined.
            
            We first prove the injectivity. Suppose we have representatives $\phi_i:[T_i/G]\to \cM, \psi_j:[T_j/G] \to \cM$ of two elements in $\colim \Mor_S([T_i/G],\cM)$ such that $\phi_i \pi_i = \psi_j \pi_j$. Choose $k \in I$ such that $i \le k$ and $j \le k$ and consider the following commutative diagram
            \begin{center}
                \begin{tikzcd}
                    G \times_S T \ar[r,shift left] \ar[r,shift right] \ar[d] \ar[rd,phantom,"\square"]& T \ar[r] \ar[d] \ar[rd,phantom,"\square"] & \left[ T/G \right] \ar[d,"\pi_k"] \\
                    G \times_S T_k \ar[r,shift left] \ar[r,shift right] & T_k \ar[r] \ar[rd,shift left] \ar[rd,shift right]& \left[ T_k/G \right] \ar[d,shift left,"\psi_k"] \ar[d,shift right,"\phi_k"']\\
                    & & \cM
            \end{tikzcd}
            \end{center}
            where the squares are all Cartesian. Then $\phi_k\pi_k = \psi_k \pi_k$ so that the four compositions $G \times_S T \rightrightarrows T \to T_k \rightrightarrows \cM$ are all equal. Denote this map by $\theta$. Note that $G \times_S T = \lim (G \times_S T_i)$ as fiber product is a limit and limits commute with limits. Since all $G \times_S T_i$ and $G \times_S T$ are affine schemes, we have $\colim \Mor_S(G \times_S T_i,\cM) \simeq \Mor(G \times_S T,\cM)$ and hence we can find a representative $\theta_l:G \times_S T_l \to \cM$ of the element in $\colim \Mor_S(G \times_S T_i, \cM)$ corresponding to $\theta$. If we choose $m \in I$ such that $k \le m$ and $l \le m$, then we have a commutative diagram as above with index $m$ and the four compositions $G \times_S T_m \rightrightarrows T_m \rightrightarrows \cM$ must be all equal to $\theta_m$ since precomposing them with $G \times_S T \to G \times_S T_m$ all give $\theta$. By descent, we have $\phi_m = \psi_m$ so that two representatives we have chosen at the beginning represent the same element in $\colim \Mor_S([T_i/G],\cM)$.
            
            For the surjectivity, suppose we are given a morphism $\phi:[T/G] \to \cM$ and consider the following diagram
            \begin{center}
                \begin{tikzcd}
                    G \times_S T \ar[r,shift left] \ar[r,shift right] \ar[d] \ar[rd,phantom,"\square"]& T \ar[r] \ar[d] \ar[rd,phantom,"\square"] & \left[ T/G \right] \ar[d,"\pi_i"] \ar[dd,bend left = 40, "\phi"] \\
                    G \times_S T_i \ar[r,shift left] \ar[r,shift right] & T_i \ar[r] \ar[rd]& \left[ T_i/G \right] \\
                    & & \cM
                \end{tikzcd}
            \end{center}
             which holds for all $i \in I$ with the squares all Cartesian. Fix $i \in I$. From the diagram, we see that two compositions $G \times_S T \rightrightarrows T \to T_i \to \cM$ are equal. Denote this map by $\theta$. As in the proof of the injectivity, we can find a representative $\theta_j:G \times_S T_j \to \cM$ of an element in $\colim \Mor_S(G \times_S T_i,\cM)$ corresponding to $\theta$. Then two compositions $G \times_S T_j \rightrightarrows T_j \to \cM$ must be the same since precomposing them with $G \times_S T \to G\times_S T_j$ all give $\theta$. By descent, we get a morphism $\phi_j:[T_j/G] \to \cM$ and an element of $\colim \Mor_S([T_i/G],\cM)$ represented by $\phi_j$ is mapped to $\phi$ under the natural map $\colim \Mor_S([T_i/G],\cM) \to \Mor_S ([T/G], \cM)$.
        \end{proof}
        
            Now we finish the proof of Theorem \ref{thm:main}. Continuing the discussion at the beginning of this section, replace $X$ by $X'$ which is obtained by a finite sequence of blow-ups and assume that $X \setminus U = \cup_{i=1}^N D_i$ is a simple normal crossings divisor with irreducible components $D_i$. 
            
            For each $i = 1,\cdots, N$, let $\eta_i \in X$ be the generic point of $D_i$. Then we have a following commutative diagram
            \begin{center}
                \begin{tikzcd}
                  \Spec K \arrow[d,hook] \arrow[r] \ar[rd,phantom,"\square"] &  U \arrow[r] \arrow[d,hook]& \cM \arrow[d] &\\
                 \Spec \cO_{X,\eta_i} \arrow[r]& X \arrow[r]& M. &
                \end{tikzcd}
            \end{center}
            where the left square is Cartesian, $K$ is the function field of $X$, and $\cM \to M$ is a coarse moduli space of the tame stack $\cM$. By the valuative criterion for proper morphisms of tame stacks \cite[Theorem 3.1]{BrescianiVistoliValCrit}, there exists a minimal positive integer $r_i$ which gives a representable lifting $\sqrt[r_i]{D_i'/\Spec \cO_{X,\eta_i}} \to \cM$ of the morphism $\Spec \cO_{X,\eta_i} \to X \to M$ where $D_i'$ is the pullback of $D_i$ in $\Spec \cO_{X,\eta_i}$. By the observation in the Remark \ref{rmk:root-stack-as-a-limit}, $\sqrt[r_i]{D_i'/\Spec \cO_{X,\eta_i}} = \lim_{\eta_i \in V} \sqrt[r_i]{D_i \cap V/V}$ where the limit runs over all open neighborhoods $V\subseteq X$ of $\eta_i$. Then, by applying Lemma \ref{lem:loc-fp} to the smooth presentation of root stacks, we obtain a morphism $\cV_i:=\sqrt[r_i]{D_i \cap V_i/V_i} \to \cM$, for some open neighborhood $V_i \subseteq X$ of $\eta_i$ . Shrinking $V_i$ if necessary, we may assume that $V_i \cap D_j =  \emptyset$ for all $j \ne i$.
                     
            Running this algorithm on all $D_i$'s, we get a collection of morphisms $\{\cV_i \to \cM\}_{i = 1}^N$ and an $N$-tuple of positive integers $\bm{r} =(r_1,\cdots,r_N)$ with each $r_i$ minimally chosen for $D_i$. We set $\cX = \sqrt[\bm{r}]{\bm{D}/X}$ where $\bm{D} = (D_1,\cdots,D_N)$. Each $\cV_i$ is an open substack of $\cX$, and we can identify $U$ as an open substack of $\cX$ as well since $\cX \to X$ is an isomorphism over $U$, in particular, $U$ has no stacky points. Then, the morphisms $\cV_i \to \cM$ and $U \to \cM$ agree on the overlaps since the overlaps of $\cV_i$'s are also contained in $U$ and their restrictions to the overlaps are equal to the restriction of the given morphism $U \to \cM$. Thus, they glue to a morphism $\cV:=(\cup_i \cV_i) \cup U \to \cM$.
            
            Then $\cV$ is a dense open substack of $\cX$ whose complement consists of finitely many closed points. Indeed, if $\pi: \cX \to X$ is the canonical morphism and $\cD_i = \pi^{-1}(D_i)_\text{red}$, then each $\cV_i$ contains the generic point of $\cD_i$. Each $\cD_i$ is regular, of dimension $1$ and furthermore, quasi-compact since we assumed $X$ is noetherian. Hence, each $\cD_i\setminus \cV$ consists of finitely many closed points and the same holds for $\cX \setminus \cV$. Since $\cX$ is tame (Remark \ref{rmk:root-stacks-are-tame}) and regular (Proposition \ref{prop:regularity-of-root-stack}), the morphism $\cV \to \cM$ extends to a morphism $\cX \to \cM$ by the Purity Lemma (Lemma \ref{lem:purity-tame-stack}). The uniqueness of $\cX$ and $\cX \to \cM$, as well as the compatibility of $\cX \to  \cM$ follow again from the construction and Proposition \ref{prop:separated} since $\cX$ is tame and regular, and $\cM$ is separated.
            
            \subsection{Sketch of the proof of Theorem \ref{thm:main-threefold}}\label{subsec:proof-threefold}
                As in the proof of Theorem \ref{thm:main}, we may reduce to the case where $X$ is noetherian and irreducible of dimension $3$. By the Purity Lemma (Lemma \ref{lem:purity-higher-dimensional-II} and Remark \ref{rmk:purity-regular}), we may also assume that $X\setminus U$ consists of irreducible components of codimension $1$. Since the indeterminacy locus has dimension $\le 2$, we can resolve its singularities (\cite[Corollary~0.4]{CossartJannsenSaito} or \cite[Theorem~5.9]{CossartJannsenSaito}). That is, there exists a projective, surjective morphism $\pi\colon X'\to X$, which is an isomorphism over $U$, such that $\pi^{-1}(X\setminus U)_{\mathrm{red}}$, with the reduced induced closed subscheme structure, is a simple normal crossings divisor on $X'$. Write $X'\setminus U=\bigcup_{i=1}^N D_i$. As in the proof of Theorem \ref{thm:main}, we construct local extensions of the given rational map from the root stacks of the spectra of the local rings at the generic points of the $D_i$ (which are DVRs). By Lemma \ref{lem:loc-fp}, we can spread out these maps to open neighbourhoods, and these glue to yield a morphism $\cV\to\cM$, where $\cV\subseteq\cX$ is an open dense substack. By the Purity Lemma (Lemma \ref{lem:purity-tame-stack}) again, this extends to a morphism $\cX\to\cM$, and the uniqueness and compatibilities of $\cX\to\cM$ follow from Proposition \ref{prop:separated}.

    \subsection{Sketch of the proof of Theorem \ref{thm:main-higher-dim}}\label{subsec:proof-higher-dim}
        Over a field of characteristic $0$, the reduction to the case where $f\colon\cM\to\cN$ is a coarse moduli space remains valid for arbitrary dimension $n\ge 2$, since the resolution of singularities for the graph $\Gamma$ of a rational map $X\dashrightarrow M_X$ in Section \ref{subsec:proof-reductions} holds for schemes locally of finite type and of arbitrary dimension \cite{Hironaka}. After replacing $f\colon\cM\to\cN$ by its coarse moduli space $f\colon\cM\to M$, we may again reduce to the case where $X$ is noetherian and irreducible of dimension $n$. Then an argument similar to that in Section \ref{subsec:proof-threefold} (the proof of Theorem \ref{thm:main-threefold}), together with embedded resolution of singularities in arbitrary dimension \cite{Hironaka}, proves the theorem.

\section{Applications}\label{sec:applications}

    \subsection{Morphisms at the generic point of regular varieties}\label{subsec:extension-generic}
        In this section, we resolve the indeterminacy of morphisms defined at the generic point of regular surfaces using Theorem \ref{thm:main}, and higher dimensional varieties using Theorems \ref{thm:main-threefold} and \ref{thm:main-higher-dim}.

        \begin{corollary}\label{cor:morphisms-at-the-generic-point} Let $f:\cM \to \cN$ be a morphism of stacks as in Theorem \ref{thm:main} (or Theorem \ref{thm:main-higher-dim}, or a coarse moduli space as in Theorem \ref{thm:main-threefold}, respectively). Let $X$ be a regular separated scheme as in the corresponding theorem, and $K = K(X)$ its function field. Suppose we have a commutative diagram
            \begin{center}
                \begin{tikzcd}
                    \Spec K \arrow[r] \arrow[d]& \cM \arrow[d,"f"] \\
                    X \arrow[r]& \cN
                \end{tikzcd}
            \end{center}
        Then, there exist a regular algebraic stack $\cX$ and a birational morphism $\cX \to X$ which is an isomorphism over $\Spec K$, and a morphism $\cX \to \cM$ which makes the following diagram commute.
            \begin{center}
                \begin{tikzcd}
                   & \Spec K \arrow[ld] \arrow[r] \arrow[d]& \cM \arrow[d,"f"] \\
                \cX \arrow[r] \arrow[rru,dashed] & X \arrow[r]& \cN
                \end{tikzcd}
            \end{center}
        Moreover, the morphism $\cX \to X$ factors as
        \begin{align*}
            \cX = \sqrt[\bm{r}]{\bm{D}/X'} \to X' \to X
        \end{align*}
        where $X' \to X$ is a proper birational morphism with a simple normal crossings divisor $\cup_{i=1}^N D_i$, and $\sqrt[\bm{r}]{\bm{D}/X'} \to X'$ is a root stack morphism with $\bm{D} = (D_1,\cdots,D_N)$ and $\bm{r}$ an $N$-tuple of positive integers. Furthermore, any two such $\cX \to X$ and liftings are dominated by a third.
        \end{corollary}

        \begin{proof}
            Note that $\Spec K$ is not an open subscheme of $X$ so that we cannot apply our main theorems directly. However, we have $K = \cO_{X,\eta} = \colim_{\eta \in U}\cO_X(U)$ where $\eta \in X$ is the generic point and the limit runs over the affine open neighborhood of $\eta$. Hence, $\Spec K = \lim_{\eta \in U} U$ and by Lemma \ref{lem:loc-fp}, we can find a representative $U \to \cM$ of an element in $\colim_{\eta \in U} \Mor_S(U,\cM)$ corresponding to the morphism $\Spec K \to \cM$. Thus, we have a commutative diagram
            \begin{center}
                \begin{tikzcd}
                    U \ar[rr,bend left=20] \ar[d,hook]& \Spec K \ar[l] \ar[r] \ar[d]& \cM \ar[d]\\
                    X \ar[r,equal]& X \ar[r]& \cN
                \end{tikzcd}
            \end{center}
            Then, by Theorem \ref{thm:main} (resp. Theorems \ref{thm:main-threefold}, \ref{thm:main-higher-dim}), there exists a regular algebraic stack $\cX$ with morphisms $\cX \to X$, which is an isomorphism over $U$, and $\cX \to \cM$ fitting into the diagram. Moreover, the morphism $\cX \to X$ factors as $\cX = \sqrt[\bm{r}]{\bm{D}/X'} \to X' \to X$ where $X' \to X$ is a proper birational map, $\bm{r}$ is an $N$-tuple of positive integers and effective Cartier diviosrs $\bm{D} = (D_1,\cdots,D_N)$ form a simple normal crossings divisor on $X'$.

            If we have another data of a lifting $\cX'\to \cM$ constructed from a different representative $U' \to \cM$, then we can construct a lifting $\cX'' \to \cM$ dominating the two from the representative $U \cap U' \to \cM$.
        \end{proof}

    \subsection{Extension of torsors}\label{subsec:extension-of-torsors}
    Let $R$ be a DVR with quotient field $K$. Let $X$ be a finite type scheme faithfully flat over $R$, and $G$ a finite group scheme over $K$. Given a $G$-torsor $P \to X_K$, the problem of finding an $R$-model $\cG$ of $G$ and a $\cG$-torsor $\cP \to X$ extending $P \to X_K$ is studied in various settings since Grothendieck in \cite[Exposé X]{SGA1} for $G$ a constant finite group whose order is prime to the residue characteristic of $R$. See \cite{AnteiEmsalem}, \cite{HaiDosSantos}, \cite{Mehidi1}, and \cite{Mehidi2} for other results and a brief history of the problem. In particular, it was proved in \cite{Mehidi2} that when $G$ is a finite commutative group scheme over $K$, $G$-torsors over the generic locus of a regular curve $\cC$ over $R$ can be extended to a $\cG$-torsor over $\cC$ for some $R$-model $\cG$ of $G$.
 
 Our main theorem directly applies to extend $G$-torsors when $G$ is a finite, flat, linearly reductive group over $R$ even if $G$ is not commutative.

    \begin{corollary}\label{thm:extension-of-torsors} Let $R$ be a DVR with quotient field $K$. Let $C$ be a regular curve over $K$, and $\cC$ an $R$-regular model of $C$. Let $G$ be a finite, flat, linearly reductive group over $R$, and let $P \to C$ be an fppf $G_K$-torsor. Then, there exists a regular stack $\td{\cC}$ of dimension $2$ with a proper birational morphism $\td{\cC} \to \cC$, and an fppf $G$-torsor $\cP \to \td{\cC}$ extending $P \to C$. 
    \end{corollary}
    \begin{proof}
        Consider a commutative diagram of solid arrows
        \begin{center}
            \begin{tikzcd}
               & C \ar[r] \ar[d,hook] \ar[ld,hook,dashed]& \cB G_K \ar[r,hook] &\cB G \ar[d]\\
               \td{\cC} \ar[r,dashed] \ar[rrru,dashed]& \cC \ar[rr] & & \Spec R
            \end{tikzcd}
        \end{center}
        where the morphism $C \to \cB G_K$ classifies the given $G_K$-torsor $P \to C$. By Theorem \ref{thm:main}, there exists a regular stack $\td{\cC}$ of dimension $2$ with a proper birational map $\td{\cC} \to \cC$, and a morphism $\td{\cC} \to \cB G$ filling in the diagram. The morphism $\td{\cC} \to \cB G$ classifies a $G$-torsor $\cP \to \td{\cC}$ extending $P \to C$. 
    \end{proof}

    \subsection{Extension of fibrations}\label{subsec:fibrations}
    If $\cM$ is a tame moduli stack, for example, any Deligne–Mumford stack over a field of characteristic $0$ (such as $\overline{\cM}_{g,n}$, moduli of principally polarized abelian varieties, moduli of primitively polarized K3 surfaces, or KSBA moduli of stable varieties/pairs), or \(\overline{\cM}_{1,1}\) over $\bZ[\frac{1}{6}]$, applying our main theorems to the coarse space $\cM \to M$ yields an extension of a fibration over a dense open subset of a regular variety whose fibers are classified by the stack $\cM$. As an illustration, we treat the case of stable elliptic curves $\overline{\cM}_{1,1}$ in Corollary \ref{cor:extension-of-elliptic-fibrations}, and we give an explicit algorithm resolving the indeterminacy of a rational map in Example \ref{ex:family-of-elliptic-curves}.
    
    \begin{corollary}\label{cor:extension-of-elliptic-fibrations} Let $X$ be a regular separated scheme of dimension $\le 3$, (resp. $n \ge 2$), locally of finite type over $\bZ [ \frac{1}{6}]$ (resp. a field $k$ of characteristic $0$). Then, any elliptic fibration over a dense open subscheme $U \subseteq X$ extends to a regular stack $\cX$ with a proper birational map $\cX \to X$ which is an isomorphism over $U$.
    \end{corollary}

    \begin{proof}
        An elliptic fibration over $U \subseteq X$ corresponds to a rational map $X \dashrightarrow \ov{\cM}_{1,1}$ defined on $U$ over the coarse moduli space $\ov{\cM}_{1,1} \to \bP^1$, the $j$-invariant map. Once we resolve the indeterminacy of the rational map $X \dashrightarrow \ov{\cM}_{1,1} \to \bP^1$ of schemes by a finite sequence of blow-ups $X' \to X$, Theorems \ref{thm:main}, \ref{thm:main-threefold}, and \ref{thm:main-higher-dim} applied to the square
        \begin{center}
            \begin{tikzcd}
                U \ar[r] \ar[d,hook]& \ov{\cM}_{1,1} \ar[d]\\
                X' \ar[r] & \bP^1
            \end{tikzcd}
        \end{center}
        give the desired results.
    \end{proof}

    \begin{example}\label{ex:family-of-elliptic-curves}
    Let $k$ be the base field with $\charac k \ne 2,3$. Consider a moduli stack $\ov{\cM}_{1,1}$ of stable genus $1$ curve with a marked point whose coarse moduli space is the projective line $\bP^1$. The canonical morphism $\ov{\cM}_{1,1} \to \bP^1$ sends an isomorphism class of elliptic curves to its $j$-invariant, and stable nodal curves to $\infty \in \bP^1$.
    
    We have a rational map $\bA_{a,b}^2 \dashrightarrow \bP^1$, $(a,b) \mapsto [4a^3:4a^3+27b^2]$ with the indeterminacy locus $\{(0,0)\}$. Since any stable curve in $\ov{\cM}_{1,1}$ has a Weierstrass form, we also have a rational map $\bA_{a,b}^2 \dashrightarrow \ov{\cM}_{1,1}$, $(a,b) \mapsto y^2z = x^3 + axz^2 + bz^3$ with the marked point $[0:1:0]$, whose indeterminacy locus is $\{(0,0)\}$; the points $(a,b)$ on the locus $4a^3 + 27b^2=0$ are mapped to the stable nodal curve with a marked point. These rational maps are compatible with the canonical morphism $\ov{\cM}_{1,1} \to \bP^1$. To apply our algorithm, we first resolve the indeterminacy of the rational map $\bA_{a,b}^2 \dashrightarrow \bP^1$. Denote the divisors $\{4a^3=0\}$ and $\{4a^3+27b^2=0\}$ of $\bA_{a,b}^2$ by $C_0$ and $C_1$.

    Let $\pi_1:X_1 \to \bA_{a,b}^2$ be the blow-up at the origin, and $E_1$ the exceptional divisor. We have $X_1 = \{((a,b),[p:q]) \in \bA_{a,b}^2 \times \bP^1 \mid aq=bp\}$. $X_1$ is covered by two open charts $\{p=1\}$, $\{q=1\}$ with coordinates $(a,q)$, $(b,p)$. $\pi_1^* C_0 = C_0' + 3E_1$ and $\pi_1^* C_1 = C_1' + 2E_1$ where $C_0'$, $C_1'$ are strict transforms of $C_0$, $C_1$, and the canonical sections of the divisors in each chart are as in the table below. A blank indicates that the divisor does not intersect the chart. The linear series $|\langle C_0'+E_1,C_1' \rangle|$ defines a rational map $X_1 \dashrightarrow \bP^1$ whose indeterminacy locus is the base locus consisting of a single closed point $p_1 = ((0,0),[1:0])$ in the $(a,q)$-chart. 
    \begin{center}
        \begin{tabular}{ |c|c|c|c| } 
        \hline
             & $C_0'$ & $C_1'$ & $E_1$\\ 
        \hline
            $(b,p)$ & $4p^3$ & $4bp^3 + 27$ & $b$\\ 
        \hline
            $(a,q)$ & & $4a + 27q^2$ & $a$ \\ 
        \hline
        \end{tabular}
    \end{center}

    Let $\pi_2:X_2 \to X_1$ be the blow-up at $p_1$, and $E_2$ the exceptional divisor. $X_2$ is covered by $(b,p)$-chart and the blow-up $\{((a,q),[u:v])\mid av=qu, [u:v] \in \bP^1\}$ of $(a,q)$-chart. Moreover, the latter is covered by two open charts $\{u=1\}$, $\{v=1\}$ with cooardinates $(a,v)$, $(q,u)$. Since we took blow-up at the point in the $(a,q)$-chart and $C_0$ is contained in the $(b,q)$-chart, $\pi_2^*C_0' = C_0''$ where $C_0''$ is the strict transform of $C_0'$. It follows that $\pi_2^*(C_0'+E_1) = C_0'' + E_1' + E_2$, $\pi_2^* C_1' = C_1'' + E_2$ where $C_1'',E_1'$ are strict transforms of $C_1',E_1$, and the canonical sections of the divisors in each chart are as in the table below. The linear series $|\langle C_0''+E_1',C_1'' \rangle|$ defines a rational map $X_2 \dashrightarrow \bP^1$ whose indeterminacy locus is the base locus consisting of a single closed point $p_2 = ((0,0),[0:1])$ in the $(q,u)$-chart. 
     \begin{center}
        \begin{tabular}{ |c|c|c|c|c| } 
        \hline
             & $C_0''$ & $C_1''$ & $E_1'$ & $E_2$ \\ 
        \hline
            $(b,p)$ & $4p^3$ & $4bp^3 + 27$ & $b$ & \\ 
        \hline
            $(a,v)$ & & $4 + 27av^2$ &  & $a$\\ 
        \hline
            $(q,u)$ & & $4u + 27q$ & $u$ & $q$\\ 
        \hline
        \end{tabular}
    \end{center}

    Lastly, let $\pi_3:X_3 \to X_2$ be the blow-up at $p_2$, and $E_3$ the exceptional divisor. $X_3$ is covered by $(b,p)$, $(a,v)$-charts and the blow-up $\{((q,u),[r:s])\mid qs=ur, [r:s] \in \bP^1\}$ of $(q,u)$-chart at $p_2$. The open set obtained by the blow-up is covered by two open charts $\{r=1\}$, $\{s=1\}$ with coordinates $(q,s)$, $(u,r)$. As in the previous step, we have $\pi_3^*(C_0''+E_1') = C_0''' + E_1'' + E_3$ and $\pi_3^*(C_1'') = C_1''' + E_3$ where $C_0''',C_1''',E_1'',E_2'$ are strict transforms of $C_0'',C_1'',E_1',E_2$, and the canonical sections of the divisors in each chart are given in the table below. The linear series $|\langle C_0'''+E_1'',C_1''' \rangle|$ is base-point free and thus defines a morphism $X_3 \to \bP^1$.
      \begin{center}
        \begin{tabular}{ |c|c|c|c|c|c| } 
        \hline
             & $C_0'''$ & $C_1'''$ & $E_1''$ & $E_2'$ & $E_3$ \\ 
        \hline
            $(b,p)$ & $4p^3$ & $4bp^3 + 27$ & $b$ & &\\ 
        \hline
            $(a,v)$ & & $4 + 27av^2$ &  & $a$ &\\ 
        \hline
            $(q,s)$ & & $4s + 27$ & $s$ &  & $q$\\ 
        \hline
            $(u,r)$ & & $4 + 27r$ &  & $r$ & $u$\\ 
        \hline
        \end{tabular}
    \end{center}
    Now we have the following commutative diagram
    
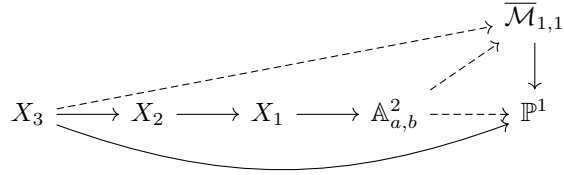
\begin{center}
        \begin{tikzcd}
            & & & & \ov{\cM}_{1,1} \ar[d] \\
            X_3 \ar[rrrru,dashed] \ar[r] \ar[rrrr,bend right = 20]& X_2 \ar[r]& X_1 \ar[r] & \bA_{a,b}^2 \ar[ru,dashed] \ar[r,dashed] & \bP^1
        \end{tikzcd}
    \end{center}
    where the indeterminacy locus of the rational map $X_3 \dashrightarrow \ov{\cM}_{1,1}$ is the preimage of $(0,0) \in \bA_{a,b}^2$, namely, $E_1''\cup E_2' \cup E_3$. In particular, it is a normal crossings divisor in $X_3$, and by choosing appropriate integers $n_1,n_2,n_3$, we can construct a morphism $\sqrt[(n_1,n_2,n_3)]{(E_1'',E_2',E_3)/X_3} \to \ov{\cM}_{1,1}$ as in the Theorem \ref{thm:main}.

   Denoting the composition $X_3 \xrightarrow{\pi_3} X_2 \xrightarrow{\pi_2} X_1 \xrightarrow{\pi_1} \bA_{a,b}^2$ by $\pi$, the sections $\pi^*a, \pi^*b \in \cO_{X_3}$ on each chart are as follows.
    \begin{center}
        \begin{tabular}{ |c|c|c| } 
        \hline
             & $\pi^*a$ & $\pi^*b$ \\ 
        \hline
            $(b,p)$ & $bp$& $b$\\ 
        \hline
            $(a,v)$ & $a$ & $a^2v$ \\ 
        \hline
            $(q,s)$ & $q^2 s$ & $q^3s$ \\ 
        \hline
            $(u,r)$ & $u^2r$& $u^3r^2$\\ 
        \hline
        \end{tabular}
    \end{center}
    Therefore, the rational map $X_3 \dashrightarrow \ov{\cM}_{1,1}$ is defined on each chart by
    \begin{align*}
        (b,p) &\mapsto y^2 z = x^3 + bpxz^2 + bz^3 \\
        (a,v) &\mapsto y^2 z = x^3 + axz^2 + a^2vz^3 \\
        (q,s) &\mapsto y^2 z = x^3 + q^2sxz^2 + q^3sz^3 \\
        (u,r) &\mapsto y^2 z = x^3 + u^2r xz^2 + u^3r^2z^3
    \end{align*}
     Note that two elliptic curves $y^2z = x^3 + axz^2 + bz^3$, $y^2z = x^3 + a'xz^2 + b'z^3$ are isomorphic if and only if $(a',b') = (t^{-4}a,t^{-6}b)$ for some $t$ \cite[Remark on page 627]{EdidinGraham}. From this, we can find $n_1,n_2,n_3$.
     
     For example, $E_1''$ intersects only two charts, $(b,p)$ and $(q,s)$-charts. On $(b,p)$-chart, letting $t = b^{1/6}$, the elliptic curve $y^2z = x^3 + bpxz^2 + bz^3$ is isomorphic to the elliptic curve 
     \begin{align*}
         y^2z = x^3 + b^{-4/6}bp xz^2 + b^{-6/6} bz^3 = x^3 + b^{1/3}pxz^2 + z^3.
     \end{align*} 
     The equation gives an elliptic curve even on $E_1''$, the locus where $b = 0$, and thus we can extend a rational map $\Spec \cO_{X_3,\eta_1} \dashrightarrow \ov{\cM}_{1,1}$ to $\sqrt[6]{\Spec \cO_{X_3,\eta_1}} \to \ov{\cM}_{1,1}$ where $\eta_1$ is the generic point of $E_1''$ by 
     \begin{align*}
         (b,p) \mapsto y^2z = x^3 + b^{1/3}pxz^2 + z^3
     \end{align*}
     
    Using a similar argument, for $\cX = \sqrt[(6,4,2)]{(E_1'',E_2',E_3)/X_3}$, we have a lifting $\cX \to \ov{\cM}_{1,1}$ of $X_3 \to \bP^1$ which resolves the indeterminacy of the rational map $X_3 \to \ov{\cM}_{1,1}$ and the lifting is defined on each chart as follows.
     \begin{align*}
        (b,p) &\mapsto y^2z = x^3 + b^{1/3}pxz^2 + z^3 \overset{\quad t=b^{1/6}}{\sim} y^2 z = x^3 + bpxz^2 + bz^3 \\
        (a,v) &\mapsto y^2z = x^3 + xz^2 + a^{1/2}vz^3 \overset{\quad t=a^{1/4}}{\sim} y^2 z = x^3 + axz^2 + a^2vz^3\\
        (q,s) &\mapsto y^2z = x^3 + s^{1/2}xz^2 + z^3 \overset{\quad t=q^{1/2}s^{1/6}}{\sim} y^2 z = x^3 + q^2sxz^2 + q^3sz^3\\
        (u,r) &\mapsto y^2z = x^3 + xz^2 + r^{1/2}z^3 \overset{\quad t=u^{1/2}r^{1/4}}{\sim} y^2 z = x^3 + u^2r xz^2 + u^3r^2z^3
    \end{align*}

    Moreover, we may produce the weighted blow-up $\cB_{(0,0)}^{(4,6)} \bA^2$ of $\bA^2$ at the origin by taking the weighted blow-down of $\cX$ contracting the divisors $E_1''$ and $E_2$. Indeed, the indeterminacy of the rational map above was also resolved in \cite{Inchiostro} by identifying $\ov{\cM}_{1,1}$ with the weighted projective space $\cP(4,6)$ and taking the weighted blow up $\cB_{(0,0)}^{(4,6)} \bA^2$. He also proved that the morphism $\cB_{(0,0)}^{(4,6)} \bA^2 \to \ov{\cM}_{1,1}$ can be identified with the forgetting morphism $\ov{\cM}_{1,2} \to \ov{\cM}_{1,1}$. 
    
    Note that the morphism $\cX \to \ov{\cM}_{1,1}$ we have constructed is not representable. Indeed, let $P$ be the closed point which is the intersection of $E_1''$ and $E_3$. Since $P \in E_1'' \cap (q,s)$-chart, and $E_1'' = \{s=0\}$ in the $(q,s)$-chart, $P$ is mapped to the point $[0:1]$ in $\ov{\cM}_{1,1}$ under the identification $\ov{\cM}_{1,1} \simeq \cP(4,6)$ with the weighted projective space. Since the stabilizer group of $P$ is $\mu_6 \times \mu_2$ and that of $[0:1]$ is $\mu_6$, we have an induced map $\mu_6 \times \mu_2 \to \mu_6$ of stabilizer groups which cannot be injective. Hence, $\cX \to \ov{\cM}_{1,1}$ is not representable at $P$. Similarly, for the closed point $Q$ which is the intersection of $E_2'$ and $E_3$, we have an induced map $\mu_4 \times \mu_2 \to \mu_4$ of stabilizer groups which cannot be injective. Hence, $\cX \to \ov{\cM}_{1,1}$ is also not representable at $Q$. 

    However, we can take relative coarse moduli space $\cX' \to \ov{\cM}_{1,1}$ of $\cX \to \ov{\cM}_{1,1}$ which is a representable map. For example, at $P \in E_1''' \cap E_3$, the induced map on the stabilizer groups is defined by
    \begin{align*}
        \mu_6 \times \mu_2 \to \mu_6, \quad (\zeta_6,1) \mapsto \zeta_6,\,\, (1,\zeta_2) \mapsto \zeta_6^3 = -1
    \end{align*}
    So the kernel is $\mu_2$, generated by $(\zeta_6^3,\zeta_2)$, and we have an isomorphism
    \begin{align*}
        \mu_6 \times \mu_2 / \mu_2 \simeq \mu_6
    \end{align*}
    where the quotient $\mu_6 \times \mu_2 / \mu_2$ is generated by the image of $(\zeta_6,1)$. Over an open neighborhood $\cU \subseteq \cX$ of $P$, we have $E_1'' = \{s=0\}$ and $E_3=\{q=0\}$ so that
    \begin{align*}
        \cU = [(\Spec \cO_{\cU}[z,w]/(z^6-s,w^2-q))/\mu_6 \times \mu_2]
    \end{align*}
    Then, the relative coarse moduli space is defined by
    \begin{align*}
        \cU' &= [(\Spec \cO_{\cU}[z,w]/(z^6 -s,w^2-q))/\mu_2) / (\mu_6\times \mu_2/\mu_2)] \\
        &\simeq [(\Spec \cO_{\cU}[A,B,C]/(A^3 -s, AC - B^2, C-q) )/ \mu_6] 
    \end{align*}
    where $A = z^2$, $B = zw$, and $C = w^2$. Similarly, we can locally construct the relative coarse space and glue them together to the global relative coarse space $\cX'$ and a representable morphism $\cX' \to \ov{\cM}_{1,1}$.
    \end{example}
    
	\bibliographystyle{alpha}
	\bibliography{reference}
	
\end{document}